\documentclass{amsart}
\usepackage{amsfonts,amsthm,amsmath}
\usepackage{amssymb}
\usepackage{amscd}
\usepackage[utf8]{inputenc}
\usepackage[a4paper]{geometry}
\usepackage{enumerate}
\usepackage[usenames,dvipsnames]{color}
\usepackage{array}
\usepackage{array,tabularx}

\numberwithin{equation}{section}
\numberwithin{equation}{subsection}

\theoremstyle{plain}

\newtheorem{theorem}[equation]{Theorem}
\newtheorem{lemma}[equation]{Lemma}
\newtheorem{proposition}[equation]{Proposition}

\newtheorem{corollary}[equation]{Corollary}

\newtheorem{lem}[equation]{Lemma}

\theoremstyle{definition}
\newtheorem{example}[equation]{Example}
\newtheorem{remark}[equation]{Remark}

\newtheorem{definition}[equation]{Definition}

\newtheorem{bekezdes}[equation]{}

\DeclareMathOperator{\di}{div}

\DeclareMathOperator{\brr}{\bar{r}}
\DeclareMathOperator{\br}{\bar r}

\def\C{\mathbb C}
\def\Q{\mathbb Q}

\def\Z{\mathbb Z}

\def\im{{\rm im}}

\newcommand{\calv}{{\mathcal V}}

\newcommand{\calO}{{\mathcal O}}

\newcommand{\calS}{{\mathcal S}}

\newcommand{\calL}{\mathcal{L}}

\newcommand{\tX}{\widetilde{X}}

\newcommand{\reg}{{\rm reg}}

\newcommand{\cO}{{\mathcal O}}

\newcommand{\eca}{{\rm ECa}}
\newcommand{\pic}{{\rm Pic}}

\newcommand{\m}{\mathfrak{m}}

\newcommand{\bZ}{{\mathbb{Z}}}

\author{J\'anos Nagy}
\address{Alfr\'ed R\'enyi Institute of Mathematics,
Re\'altanoda utca 13-15, H-1053, Budapest, Hungary}
\email{nagy.janos@renyi.hu}

\author{Andr\'as N\'emethi}
\address{Alfr\'ed R\'enyi Institute of Mathematics,
Re\'altanoda utca 13-15, H-1053, Budapest, Hungary \newline
 \hspace*{4mm} ELTE - University of Budapest, Dept. of Geometry, Budapest, Hungary \newline \hspace*{4mm}
BCAM - Basque Center for Applied Math.,
Mazarredo, 14 E48009 Bilbao, Basque Country – Spain}
\email{nemethi.andras@renyi.mta.hu }

\author{Tomohiro Okuma}
\address{Yamagata University, Dept. of Math. Sciences, Yamagata, 990-8560, Japan}
\email{okuma@sci.kj.yamagata-u.ac.jp}

\title{Normal reduction number of normal surface singularities}

\begin{document}

\keywords{normal surface singularities, links of singularities,
plumbing graphs, rational homology spheres, geometric genus, line bundles, Picard group, reduction of an ideal,
normal reduction number}

\subjclass[2010]{Primary. 32S05, 32S25, 32S50, 57M27
Secondary. 14Bxx, 14J80, 57R57}
\thanks{The first two  authors are partially supported by NKFIH Grant ``\'Elvonal (Frontier)'' KKP 126683.}
\begin{abstract}
Let $(X,o)$ be a complex analytic normal surface singularity and let $\calO_{X,o}$ be its local ring.
We investigate the normal reduction number of $\calO_{X,o}$ and related numerical analytical invariants
via resolutions $\tX\to X$ of $(X,o)$ and cohomology groups  of different line bundles $\calL\in\pic(\tX)$.
The normal reduction number is the universal optimal bound from
which powers of certain ideals have stabilization properties.
Here we combine this with stability properties of the iterated Abel maps. Some of the main results
provide topological upper bounds for both stabilization properties.

 The present note was partially motivated by the open problems formulated in \cite{O19}.
Here we answer several of them.
\end{abstract}

\maketitle

\linespread{1.2}


\pagestyle{myheadings} \markboth{{\normalsize  J. Nagy}, {\normalsize  A. N\'emethi}, {\normalsize  T.Okuma}}
{{\normalsize Normal reduction number}}


\section{Introduction}\label{s:intr}

Let $(X,o)$ be a complex analytic normal surface singularity and let $\calO_{X,o}$ be its local ring.
Our goal is to investigate certain properties of the normal reduction number of $\calO_{X,o}$
via certain  numerical invariants associated with
 resolutions $\tX\to X$ of $(X,o)$ and  with the cohomology groups  of different line bundles $\calL\in\pic(\tX)$.

\subsection{Some ring--theoretical invariants of $\calO_{X,o}$} \label{ss:1.1} (See e.g. \cite{HS,O19}.)
 Let $I$ be an $\m$--primary ideal of $\cO_{X,o}$, where $\m$ is the maximal ideal.
The integral closure $\bar{I}$ of $I$ is the ideal consisting of all solutions of equations of type
$z ^n +c_1 z^{n-1} + \dots+ c_{n-1} z + c_n=0$ with coefficients $c_i \in I^i $.
Then $ I \subseteq \bar{I} \subseteq \sqrt{I}$.
We say that $I$ is \textit{integrally closed} if $I = \bar{I}$.
In the sequel   we will assume that $I=\bar{I}$.

Recall that an ideal $J\subset I$ is called a `reduction' of $I$ if $I$ is integral over $J$, or, equivalently,
$I^{r+1}=JI^r$ for a certain $r$. An ideal $Q\subset I$ is called `minimal reduction' of $I$ if $Q$ is minimal
among the reductions of $I$. In the case of $\m$--primary (integrally closed) ideals $I$ of $\calO_{X,o}$,
a minimal reduction is a parameter ideal.

For any  minimal reduction $Q$ of $I$, we have that ${\overline{I^{n+1}}} = Q {\overline{I^{n }}} $ for all large $n$.
We define the \textit{normal reduction number} $\br(I)$ of $I$ by $\br(I):= \min\{ r \;|\; \overline{I^{n+1}} = Q \overline{I^n}\;
\text{for all} \; n \ge r\ \mbox{and a minimal reduction $Q$}\}$;
this integer  does not depend on the choice of $Q$ (cf. \cite{Hu1987,OWY19a}).
The \textit{normal reduction number of} $(X,o)$ is defined by
$$\brr(X,o):=\max\{ \br(I) \;|\; I\subset \cO_{X,o}, \; \sqrt{I}=\m, \; I=\bar I\}.$$

Since the (normal)  filtration $\{\overline{I^n}\}$ contains several key information about the blow-up of $(X,o)$
(along $I$), we expect that the normal reduction numbers should encode  important information
about the ring $\cO_{X,o}$ and also about the resolution spaces of $(X,o)$. However,
the structure of the set of  reductions   is clarified only for very special cases
(see e.g.  \cite{OWY19a,OWY19b,O19}).

Some major questions are the following:

(i) For any fixed  $(X,o)$ find the possible values of $\brr(I)$ and $\brr(X,o)$.
(The expectation is that in general the answer depends essentially on the analytic type of $(X,o)$.)

(ii) Fix a topological type of a singularity. Show that the integers $\brr(I)$ and $\brr(X,o)$,
associated with  all the analytic structures supported on the fixed topological type,  can universally be bounded (from above)
by a topological invariant.

(iii) Find an optimal topological upper bound of (ii) (which is realized by a certain analytic structure).

\subsection{Some singularity--theoretical invariants of $(X,o)$}
Fix a resolution  $\phi:(\tX,E)\to (X,o)$
and an effective cycle $Z\geq E$ supported on the exceptional curve $E$, and $l'\in H^2(\tX,\Z)$ a Chern class
with $(-l',E_v)\leq  0$ for any irreducible component $E_v$ of $E$ (denoted as $-l'\in\calS'$, cf. \ref{ss:topol}).

If  $\calL\in \pic^{l'}(Z)$ is a line bundle on $Z$ with Chern class $l'$ and  without fixed components then
$h^1(Z,\calL^{n})\geq h^1(Z,\calL^{n+1})$ for $n\geq 0$ (cf. section \ref{s:STAB}).
Hence the sequence $n\mapsto h^1(Z,\calL^{ n})$
is non--increasing  and it is constant for $n\gg 0$.
For such an $\calL$ we define  $n_0(Z,\calL)$ as  the smallest
integer $n$ such that $h^1(Z,\calL^{n})= h^1(Z,\calL^{n+1})$.
In fact,   $n_0(Z,\calL)=
\min\{n\ge 0 \,:\, h^1(Z,\calL^n)=h^1(Z,\calL^{m}), \forall m>n\}$ as well, see e.g. \cite[Lemma 3.6]{O17}.

Let us consider next another  related stability problem as well.

For fixed $Z$ and $l'$ as above, one can consider $\eca^{l'}(Z)$, the space of
 effective Cartier divisors with
Chern class $l'\in L'$ and the Abel map
$c^{l'}(Z): \eca^{l'}(Z)\to \pic^{l'}(Z)$.
 Note that $\calL\in \im (c^{l'}(Z))$ if and only if $\calL$ has no fixed components,  cf. \ref{ss:4.1}.

Then  one shows that  the sequence  $n\mapsto \dim(\im(c^{nl'}(Z)))$ is a non--decreasing and it is
bounded from above by $h^1(\calO_Z)$, hence it must stabilise, cf. \ref{bek:I}.
 Let $n_0'(Z,l')$ be the smallest
integer $n$ such that  $\dim(\im(c^{nl'}(Z)))= \dim(\im(c^{(n+1)l'}(Z)))$.
Again,
$n_0'(Z,l')=
\min\{ n\ge 0\,:\,
\dim(\im(c^{nl'}(Z)))= \dim(\im(c^{ml'}(Z))), \forall m>n\}$, cf. Remark \ref{rem:stab}.

The integers $n_0(Z,\calL)$ and $n_0'(Z,l')$ are important invariants of the singularity.
If $Z\gg 0$ then they are independent of $Z$, and they will be denoted by $n_0(\calL)$ and $n_0'(l')$ respectively.

In Lemma \ref{lem:COMP} we show that for any $Z$ and any $-l'\in \calS'$ as above
 $$n_0'(Z,l')\leq \min\, \{\,  n_0(Z,\calL)\, :\, \calL\in \im (c^{l'}(Z)).\}$$

 In parallel to the questions formulated in \ref{ss:1.1} we can formulate the very same type of questions, but now for
 the integers $h^1(Z, \calL)$, $n_0(Z,\calL)$ and $n_0'(Z,l')$.
 (Note that one of the goals of the theory of Abel maps is also to understand the possible values of $h^1(Z, \calL)$
 and the corresponding stratification in ${\rm Pic}(Z)$ induced by $h^1(Z, \calL)$, cf. \cite{NNI,NNII,NNIII}.)

\subsection{The connection between the two approaches}
By a result of Lipman \cite{Li},
if $I$ is an integrally closed $\m$--ideal then  there exist a resolution $\phi: \tX\to X$ and an integral
 cycle $l\in\calS'$ on $\tX$ such that $I = (\phi_*\mathcal{O}_{\tX}(-l))_o=H^0(\tX,\mathcal{O}_{\tX}(-l))$  and $I\mathcal{O}_{\tX}=\mathcal{O}_{\tX}(-l)$ (that is, $\calO_{\tX}(-l)$ is generated by global sections).
Then we can define for every integer $n \ge 0$  a non-increasing  chain of integers
$q(nI):= h^1(\mathcal{O}_{\tX}(-nl))$, where $q(0I) := p_g(X,o)$ is the geometric genus of $(X,o)$. It turns out that  $q(nI)$ is independent of the representation
of $I$ as $\mathcal{O}_{\tX}(-l)$ (cf. \cite[\S 3]{OWY14}).
Furthermore,  we have $0 \le q(I) \le p_g(X,o)$ and
$\br(I) =\min\{n \in \Z_{> 0} \,|\, q((n-1)I)=q(nI)  \}$ (cf. \cite{OWY15a}).

This makes the bridge between the normal reduction number of ideals and the study of the cohomology of line bundles on resolution spaces. In particular,  $\br(I)-1=n_0(\calO_{\tX}(-l))$.

As we mentioned above, the normal reduction number $\brr(X,o)$ plays a key role in the algebraic study of the ring
$\mathcal{O}_{X,o}=(\phi_*\mathcal{O}_{\tX})_o$. Here are some
exemplifications of certain  key bounds:

\vspace{1mm}

$\bullet$  The natural homomorphism $\phi_*\cO_{\tX}(-nl) \otimes \phi_*\cO_{\tX}(-l) \to \phi_*\cO_{\tX}(-(n+1)l)$
is surjective for $n\ge \brr(X,o)$. That is, the graded algebra $\bigoplus _{n\ge 0}\phi_*\cO_{\tX}(-nl)$ is generated by parts of degree $\le \brr(X,o)$.

$\bullet$
The function $\varphi(n):=\dim_{\C}H^0(\cO_{\tX})/H^0(\cO_{\tX}(-nl))$  is a polynomial function of $n$ for $n\ge \brr(X,o)$; in fact, $\varphi(n)=\chi(\cO_{nl})+h^1(\cO_{\tX})-h^1(\cO_{\tX}(-nl))$ by Kato's Riemann-Roch Theorem.


$\bullet$ In \cite{OWY15b} is proved that  $ \brr(X,o)=1$ if and only if $(X,o)$ is rational. Furthermore,
the third author proved that  $\br(X,o)= 2$ for an elliptic singularity \cite{O17}.

\subsection{The new results} The present note was partially motivated by the open problems formulated in \cite{O19}.
Here we answer several of them.
In parallel we answer also some of the questions formulated in the previous subsections.
The main results are the following.

\vspace{1mm}

\noindent {\bf Theorem A.} {\it
 Fix a  complex analytic normal surface singularity $(X, o)$ such that all the irreducible exceptional curves (in some resolution)
 are rational.
 Fix also an arbitrary integer $0 \leq q \leq p_g(X,o)$.
Then there exists a resolution $\tX \to X$ and an effective integral cycle $l > 0$ such that $\calO_{\tX}(-l)$ is base point free and $h^1(\calO_{\tX}(-l)) = q$.}

This answers Conjecture 2.8 of \cite{O19}. For a slightly more general statement see Remark \ref{rem:GEN}.

\vspace{1mm}

\noindent {\bf Theorem B.} {\it
Assume that $(X,o)$ is  a normal surface singularity whose link is a rational homology sphere and let
us fix  a resolution $\tX \to X$ and an effective  cycle $Z\geq E$ and $-l'\in\calS'$.
Assume that $\calL\in\pic^{l'}(Z)$   is a base point free line bundle on $Z$. Then
$n_0(Z,\calL)\leq  1 - \min_{Z\geq l>0} \chi(l)$ (where $\chi$ is the Riemann--Roch expression of cycles supported on $E$, cf. \ref{ss:topol}).
In particular, $ \brr(X,o)\leq  2 - \min_{l>0} \chi(l)$.}

This answers Problem 3.13 of \cite{O19}.

Since $\min_{l>0} \chi(l)$ is a topological invariant independent of the choice of the resolution,
the above inequality provides a topological upper bound for $ \brr(X,o)$
(valid for any analytic structure  supported on a fixed topological type).

\vspace{1mm}

The analogue of Theorem B for the images of the Abel maps is the following.

\vspace{1mm}

\noindent {\bf Theorem C.} {\it Under the assumptions of Theorem B we also have
$n_0'(Z,l')\leq  1 - \min_{Z\geq l>0} \chi(l)$. }

\vspace{1mm}

Finally, we provide a criterion which guarantees that
$n_0(Z,\calL)=n_0'(Z,l')=1$ and   $\brr(X,o)=2$.

  In this way we are able to construct non--elliptic singularities $(X,o)$  with   $\brr(X,o)=2$;
  this answers negatively Problem 3.12 of \cite{O19} (which asked whether $\brr(X,o)=2$ characterizes  the elliptic germs).

This also shows that if we fix a positive integer $k\geq 2$,
the classification of those analytic structures which satisfy $\brr(X,o)=k$
can be a very hard task.

\subsection{} Most of the techniques and the guiding ideas of the proofs are based on the theory of
Abel maps developed by the first two authors in \cite{NNI,NNII,NNIII}.

\section{Preliminaries}\label{s:prel}

\subsection{The resolution}\label{ss:notation}
Let $(X,o)$ be the germ of a complex analytic normal surface singularity,
 and let us fix  a good resolution  $\phi:\widetilde{X}\to X$ of $(X,o)$.
We denote the exceptional curve $\phi^{-1}(0)$ by $E$, and let $\cup_{v\in\calv}E_v$ be
its irreducible components. Set also $E_I:=\sum_{v\in I}E_v$ for any subset $I\subset \calv$.
For a cycle $l=\sum n_vE_v$ we write  $|l|=\cup_{n_v\not=0}E_v$  for its support.
For more details see \cite{trieste,NCL,Nfive}.
\subsection{Topological invariants}\label{ss:topol}
Let $\Gamma$ be the dual resolution graph
associated with $\phi$;  it  is a connected graph.
Then $M:=\partial \widetilde{X}$ can be identified with the link of $(X,o)$, it is also
an oriented  plumbed 3--manifold associated with $\Gamma$.
We use the same
notation $\mathcal{V}$ for the set of vertices as well.
Recall that $M$ is a rational homology sphere, if and only if
 $\Gamma$ is a tree and all the curves $E_v$ are rational.

$L:=H_2(\widetilde{X},\mathbb{Z})$, endowed
with a negative definite intersection form  $I=(\,,\,)$, is a lattice. It is
freely generated by the classes of 2--spheres $\{E_v\}_{v\in\mathcal{V}}$.
 The dual lattice $L':=H^2(\widetilde{X},\mathbb{Z})$ is generated
by the (anti)dual classes $\{E^*_v\}_{v\in\mathcal{V}}$ defined
by $(E^{*}_{v},E_{w})=-\delta_{vw}$, the opposite of the Kronecker symbol.
The intersection form embeds $L$ into $L'$. Then ${\rm Tors}(H_1(M,\mathbb{Z}))\simeq L'/L$, abridged by $H$.
Usually one also identifies $L'$ with those rational cycles $l'\in L\otimes \Q$ for which
$(l',L)\in\Z$, or, $L'={\rm Hom}_\Z(L,\Z)$.


All the $E_v$--coordinates of any $E^*_u$ are strict positive.
We define the Lipman cone as $\calS':=\{l'\in L'\,:\, (l', E_v)\leq 0 \ \mbox{for all $v$}\}$.
As a monoid,
it is generated over $\bZ_{\geq 0}$ by $\{E^*_v\}_v$.


We set  $\chi(l')=-(l',l'-Z_K)/2$, where $Z_K\in L'$ is the (anti)canonical cycle
identified by adjunction formulae
$(-Z_K+E_v,E_v)=2g_v-2$ for all $v$, where $g_v$ is the genus of $E_v$.
By Riemann-Roch theorem $\chi(l)=\chi(\calO_l)$ for any $l\in L_{>0}$.
(Here $l>0$ means that $l=\sum_vn_v E_v$ with all $n_v\geq 0$ and $l\not=0$.)

\subsection{Analytic invariants}\label{ss:analinv}
{\bf The group ${\rm Pic}(\widetilde{X})$ } is the group $H^1(\tX, \cO_{\tX}^*)$
of  isomorphism classes of analytic line bundles on $\widetilde{X}$. It  appears in the exact sequence
\begin{equation}\label{eq:PIC}
0\to {\rm Pic}^0(\widetilde{X})\to {\rm Pic}(\widetilde{X})\stackrel{c_1}
{\longrightarrow} L'\to 0, \end{equation}
where  $c_1$ denotes the first Chern class. Here
$ {\rm Pic}^0(\widetilde{X})\simeq H^1(\widetilde{X},\calO_{\widetilde{X}})/H^1(E,\Z)$, where
$ H^1(\widetilde{X},\calO_{\widetilde{X}})\simeq
\C^{p_g}$,  $p_g=p_g(X,o)$ being  the {\it geometric genus} of
$(X,o)$. $(X,o)$ is called {\it rational} if $p_g=0$.
 Artin in \cite{Artin62,Artin66} characterized rationality topologically
via the graphs; such graphs are called `rational'. By this criterion, $\Gamma$
is rational if and only if $\chi(l)\geq 1$ for any effective non--zero cycle $l\in L_{>0}$.
(Recall also that the link of any rational singularity is a rational homology sphere.)



\bekezdes
Similarly, if $Z\in L_{>0}$ is an effective non--zero integral cycle such that $Z\geq E $, and $\calO_Z^*$ denotes
the sheaf of units of $\calO_Z$, then ${\rm Pic}(Z)=H^1(Z,\calO_Z^*)$ is  the group of isomorphism classes
of invertible sheaves on $Z$. It appears in the exact sequence
  \begin{equation}\label{eq:PICZ}
0\to {\rm Pic}^0(Z)\to {\rm Pic}(Z)\stackrel{c_1}
{\longrightarrow} L'\to 0, \end{equation}
where ${\rm Pic}^0(Z)\simeq H^1(Z,\calO_Z)/H^1(E,\Z)$.
If $Z_2\geq Z_1$ then there are natural restriction maps
 ${\rm Pic}(\widetilde{X})\to {\rm Pic}(Z_2)\to {\rm Pic}(Z_1)$.
 Similar restrictions are defined at  ${\rm Pic}^0$ level too.
These restrictions are homomorphisms of the exact sequences  (\ref{eq:PIC}) and (\ref{eq:PICZ}).


We also use the notations ${\rm Pic}^{l'}(\widetilde{X}):=c_1^{-1}(l')
\subset {\rm Pic}(\widetilde{X})$ and
${\rm Pic}^{l'}(Z):=c_1^{-1}(l')\subset{\rm Pic}(Z)$
respectively. If $H^1(E,\Z)=0$ (i.e., the link is a rational homology sphere) then
they are affine spaces associated with the vector spaces
${\rm Pic}^0(\widetilde{X})$ and ${\rm Pic}^0(Z)$ respectively.



As usual, we say that $\calL\in \pic(Z)$ has no fixed components if
\begin{equation}\label{eq:H_0}
H^0(Z,\calL)_{\reg}:=H^0(Z,\calL)\setminus \bigcup_v H^0(Z-E_v, \calL(-E_v))
\end{equation}
is non--empty.
\subsection{Cohomological cycles} \cite{MR}\label{ss:COHCYC}
Assume that $(X,o)$ is  a non--rational singularity and let
$\phi:\tX\to X$ be one of its resolutions.  Then, by definition,
the cohomological cycle $ Z_{coh}$  is the (unique) minimal cycle $Z\in L_{>0}$ with $h^1(\calO_Z)=p_g(X,o)$.
This is equivalent with $h^1(\calO_{Z_{coh}})=p_g(X,o)$ and
 $h^1(\calO_{Z'})<p_g(X,o)$ whenever $Z'\not\geq Z_{coh}$.
If $p_g(X,o)=0$ then,  by definition,  we set $Z_{coh}=0$.

\subsection{$q$--cohomological cycles \cite{OWY15a,OWY15b,O17}}\label{ss:qcohcyc}
Fix a resolution $\tX\to X$.

A cycle $C\in L_{>0}$ is called a {\em $q$-cohomological cycle} if
$$
 h^1(\cO_C)=q=\max_{D>0, \,|D|\le C}h^1(\cO_D),
$$
 and $h^1(\cO_{C'})<q$ for every nonzero effective cycle $C'<C$.
This basically says that $C$ is the cohomological cycle on its support $|C|$ and the geometric genus
on this support is $q$. Note that in general, $q$-cohomological cycle on a resolution is not unique.

\subsection{Laufer's Duality}\label{ss:LauferD}

Let us fix a good resolution $\tX\to X$  as above. Then there exists a perfect pairing
(cf. \cite{Laufer72,Laufer77,NNI})
\begin{equation}\label{eq:LD}
\langle\,,\,\rangle :H^1(\tX,\cO_{\tX})\otimes  \big(
H^0(\tX\setminus E,\Omega^2_{\tX})/ H^0(\tX,\Omega^2_{\tX})\big) \longrightarrow\ \C.\end{equation}
Here  $H^0(\tX\setminus E,\Omega^2_{\tX})$ can be replaced by
$H^0(\tX,\Omega^2_{\tX}(Z))$ for $Z\gg 0$ (e.g. for any $Z$ with $Z\geq \lfloor Z_K \rfloor$),
cf. \cite[7.1.3]{NNI},  and for such $Z\gg 0$ one also has  $H^1(Z,\cO_{Z})\simeq H^1(\tX,\cO_{\tX})$.

More generally, for any $Z>0$,
from the exact sequence $0\to \Omega^2_{\tX}\to \Omega^2_{\tX}(Z)\to \calO_Z(K_{\tX}+Z)\to 0$, vanishing
$H^1(\Omega^2_{\tX})=0$ and Serre duality $H^0(\calO_Z(K_{\tX}+Z))=H^1(\calO_Z)^*$, we obtain
$H^0(\tX,\Omega^2_{\tX}(Z))/ H^0(\tX,\Omega^2_{\tX})\simeq H^1(\calO_Z)^*$. In particular,
(see also  \cite[7.4]{NNI}) we have  a perfect pairing
\begin{equation}\label{eq:LD2}
\langle\,,\,\rangle:  H^1(Z,\cO_{Z})\otimes
\big(H^0(\tX,\Omega^2_{\tX}(Z))/ H^0(\tX,\Omega^2_{\tX})\big)\longrightarrow\ \C.\end{equation}
 Let $I\subset \calv$ be a  subset. Then (\ref{eq:LD2}) can also be applied for the cycle $Z|_{\calv\setminus I}$,
 the restriction of $Z$  to the components $\{E_v\}_{v\not\in I}$. 
Let $\Omega _{Z}(I)$ be the subspace of $H^0(\tX, \Omega^2_{\tX}(Z))/ H^0(\tX,\Omega_{\tX}^2)$ generated by differential forms which have no poles along $\cup_{v\in I}E_v\setminus \cup_{v\not\in I}E_v$.
Then (see also \cite[\S8]{NNI})
\begin{equation}\label{eq:ezlc}
h^1(\calO_{Z|_{\calv\setminus I}})=\dim \Omega_{Z}(I).
\end{equation}
By (\ref{eq:LD2}) a basis $[\omega_1],\ldots, [\omega_{h}]$ of $H^0(\Omega^2_{\tX}(Z))/H^0(\Omega^2_{\tX})$
provides $h=h^1(Z,\calO_Z)$ global  coordinates in $H^1(Z,\calO_Z)$.
In \ref{bek:transport}
(under the assumption
that the link is a rational homology sphere) we will describe the
above dualities in terms of  integrations,
cf.  \cite[\S 7]{NNI}.

\begin{lemma}\label{lem:blowupcoh} (Compare with \cite[2.6]{OWY15b}.)
Let $\phi$ be a resolution as above. Fix an irreducible exceptional divisor $E_v$ such that
the multiplicity of $Z_{coh}(\tX)$ along $E_v$ is positive.
Let $\pi:\tX_{new}\to \tX$ be the blow up of $\tX$ at a generic point of $E_v$, and let $E_{new}$ be the newly created
exceptional curve. Then $Z_{coh}(\tX_{new})=\pi^* Z_{coh}(\tX)-E_{new}$.
\end{lemma}
\begin{proof} Fix
 a generic differential form $\omega\in H^0(\tX\setminus E,\Omega^2_{\tX})$.
 By Laufer's  duality its pole is exactly $Z_{coh}(\tX)$.
(I.e.,   $\Omega^2_{\tX}(Z_{coh})$ has no fixed components, cf. \cite[2.6]{OWY15b}).)
 By a local computation, and using  the fact that we blow up a generic point of $E_v$,
 the pole of $\pi^*\omega $  is $\pi^*Z_{coh}(\tX)-E_{new}$. But $\omega$ being generic,
 $\pi^*\omega$ is a generic form at $\tX_{new}$ level, hence its pole is the
 cohomological cycle on $\tX_{new}$.
\end{proof}
\subsection{The Abel map} \label{ss:4.1}
In the proofs we will use several results from the theory of Abel maps associated with resolutions of
normal surface singularities, see \cite{NNI,NNII,NNIII,NNIV}. Next
 we recall some material from this theory needed later in the proofs.

In this subsection \ref{ss:4.1}
{\it we assume that the link is a rational homology sphere}.

Let us fix an effective integral cycle  $Z\in L$, $Z\geq E$.
Let $\eca(Z)$  be the space of effective Cartier  divisors supported on  $Z$.
Note that they have zero--dimensional supports in $E$.
Taking the class of a Cartier divisor provides  a map
$c:\eca(Z)\to \pic(Z)$, called the {\it Abel map}.
Let  $\eca^{l'}(Z)$ be the set of effective Cartier divisors with
Chern class $l'\in L'$, that is,
$\eca^{l'}(Z):=c^{-1}(\pic^{l'}(Z))$.
We consider the restriction of $c$, $c^{l'}(Z) :\eca^{l'}(Z)
\to \pic^{l'}(Z)$ too, sometimes still denoted by $c$.

The bundle  $\calL\in \pic^{l'}(Z)$ is in the image ${\rm im}(c)$ of the Abel map  if and only if
it has no fixed components, that is,  if and only if
$H^0(Z,\calL)_{\reg}\not=\emptyset$.
%

One verifies that $\eca^{l'}(Z)\not=\emptyset$ if and only if $-l'\in \calS'\setminus \{0\}$. Therefore, it is convenient to modify the definition of $\eca$ in the case $l'=0$: we (re)define $\eca^0(Z)=\{\emptyset\}$,
as the one--element set consisting of the `empty divisor'. We also take $c^0(Z)(\emptyset):=\calO_Z$. Then we have
\begin{equation}\label{eq:empty}
\eca^{l'}(Z)\not =\emptyset \ \ \Leftrightarrow \ \ l'\in -\calS'.
\end{equation}
If $l'\in -\calS'$  then
  $\eca^{l'}(Z)$ is a smooth complex irreducible quasi--projective
  variety of dimension
  \begin{equation}\label{eq:DIM}
  \dim \, \eca^{l'}(Z)=(l',Z)\end{equation}
   (see \cite[Th. 3.1.10]{NNI}). Moreover, cf.  \cite[Lemma 3.1.7]{NNI},
if $\calL\in \im (c^{l'}(Z))$
then  the fiber $c^{-1}(\calL)$
 is a smooth, irreducible quasi--projective variety of  dimension
 \begin{equation}\label{eq:dimfiber}
\dim(c^{-1}(\calL))= h^0(Z,\calL)-h^0(\calO_Z)=
 (l',Z)+h^1(Z,\calL)-h^1(\calO_Z).
 \end{equation}
 Using this,  one proves (see \cite[5.6]{NNI}) that for any $\calL\in \im (c^{l'})\subset \pic^{l'}(Z)$ one has
  \begin{equation}\label{eq:dimfiber2}
h^1(Z,\calL)\geq h^1(\calO_Z)-\dim(\im (c^{l'}(Z))),
 \end{equation}
 and equality holds whenever $\calL$ is generic in $\im(c^{l'}(Z))$.

Recall  that if all $E_v$--coefficients $(-E_v^*,Z)$ of $Z$ are very large (a fact denoted by $Z\gg 0$),
then  $Z$ constitute a  `finite  model' for $\tX$. (Note that `${\rm ECa}(\tX)$' is `undefined infinite dimensional'.)
Additionally, for $Z\gg 0$  one  also has $h^1(Z,\calL)=h^1(\tX,\calL)$ for $\calL\in {\rm Pic}(\tX)$
by Formal Function Theorem.

\bekezdes \label{bek:I}
Consider again  a Chern class $l'=\sum_{v\in\calv}a_vE^*_v\in-\calS'$ as above.
The $E^*$--support $I(l')\subset \calv$ of $l'$ is defined as $\{v\,:\, a_v\not=0\}$. Its role is the following.

Besides the Abel map $c^{l'}(Z)$ one can consider its `multiples' $\{c^{nl'}(Z)\}_{n\geq 1}$ as well.
 It turns out (cf. \cite[\S 6]{NNI}), that $n\mapsto \dim \im (c^{nl'}(Z))$
is a non-decreasing sequence, hence it becomes constant,
 and   $\im (c^{nl'}(Z))$ for $n\gg 1$  are  affine subspaces parallel to each other and
parallel with the affine closure of $\im(c^{l'}(Z))$, all of the same dimension. This common dimension  will be denoted by
 $e_Z(l')=\lim_{n\to \infty} \dim(\im (c^{nl'}(Z)))$.  It  depends only
on $I(l')$. Moreover, by \cite[Theorem 6.1.9]{NNI},
\begin{equation}\label{eq:ezl}
e_Z(l')=h^1(\calO_Z)-h^1(\calO_{Z|_{\calv\setminus I(l')}}),
\end{equation}
where $Z|_{\calv\setminus I(l')}$ is the restriction of the cycle $Z$ to its $\{E_v\}_{v\in \calv\setminus I(l')}$
coordinates.

Furthermore, for $\calL\in\pic(\tX)$  (cf. \cite[Th. 6.1.9(c)--(e)]{NNI})
\begin{equation}\label{eq:free}
\mbox{if $\calL|_Z\in {\rm im} (c^{nl'}(Z))$ ($n\gg0$, $Z\gg 0$) then }
\ \left\{ \begin{array}{l}
\mbox{$\calL$ is generated by global sections,  and} \\  h^1(\tX,\calL)
=h^1(\calO_{\tX|_{\calv\setminus I(l')}}). \end{array}\right.
\end{equation}


\bekezdes {\bf The Laufer integration.} \label{bek:transport}
Consider the following situation.
We fix a smooth point $p$ on $E$, a local bidisc $B\ni p$ with local coordinates $(x,y)$ such that $B\cap E=\{x=0\}$.
We assume that a certain form   $\omega\in H^0(\tX,\Omega^2_{\tX}(Z))$ has  local equation
$\omega=\sum_{ i\in\Z,j\geq 0}a_{i,j} x^iy^jdx\wedge dy$ in $B$.

In the same time,
we fix a divisor $\widetilde{D}$ on $\tX$, whose unique component $\widetilde{D}_1$
in $B$ has local equation $y$.
 Let $\widetilde{D}_t$ be another divisor, which
is the same as $\widetilde{D}$ in the complement of $B$ and its component $\widetilde{D}_{1,t}$
in $B$ has  local equation
$y+td(t,x,y)$.

Next,   we identify
$H^1(\tX, \calO_{\tX})$ with $\pic^0(\tX)$
by the exponential map and we consider the composition
$t\mapsto \widetilde{D}_t-\widetilde{D}
\mapsto \calO_{\tX}(\widetilde{D}_t-\widetilde{D})
\mapsto \exp^{-1} \calO_{\tX}(\widetilde{D}_t-\widetilde{D})
\mapsto \langle\exp^{-1} \cO_{\tX}(\widetilde{D}_t-\widetilde{D}),\omega\rangle$.
The next  formula makes this expression  explicit.
(Here  $B=\{|x|,\, |y|<\epsilon\}$ for a small $\epsilon$, and $|t|\ll \epsilon$.)
\begin{equation}\label{eq:Tomega11}
\langle\langle \widetilde{D}_t,\omega\rangle\rangle:= \langle\exp^{-1} \cO_{\tX}(\widetilde{D}_t-\widetilde{D}),\omega\rangle=
\int_{ \substack{|x|=\epsilon\\ |y|=\epsilon}} \log\Big(1+t\frac{d(t,x,y)}{y}\Big)\cdot
 \sum_{ i\in\Z,j\geq 0}a_{i,j} x^iy^jdx\wedge dy.\end{equation}
This restricted to any cycle $Z\gg0 $ can be reinterpreted as `$\omega$--coordinate'
of the Abel map
restricted to the path $t\mapsto D_t:=\widetilde{D}_t|_Z$
(and shifted by the image of $D:=\widetilde{D}|_Z$).
If $\omega$ has no pole along the divisor $\{x=0\}$ then
$\langle\langle \widetilde{D}_t,\omega\rangle\rangle=0$
for any path $\widetilde{D}_t$.
%

If more components of $\widetilde{D}$ are perturbed then
$\langle\langle \widetilde{D}_t,\omega\rangle\rangle$ is the  sum of such contributions.

\begin{definition}\label{not:residue}
Consider the above situation  and assume that
$\widetilde{D}_1$ has  local equation $y$.
 Then, by definition,   the {\it Leray residue }
of $\omega$ along $\widetilde{D}_1$ is  the 1--form  on $\widetilde{D}_1$
(with possible poles at $\widetilde{D}_1\cap E$) defined by
$(\omega/dy)|_{y=0}=\sum_{i\in \Z} a_{i,0}x^i dx$. We denote it by ${\rm Res}_{\widetilde{D}_1}(\omega)$.
\end{definition}
Note that if in (\ref{eq:Tomega11}) $\widetilde{D}_{1,t}=y+tx^{o-1}$ for some $o\geq 1$, then
$\frac{d}{dt}|_{t=0}\langle\langle \widetilde{D}_t,\omega\rangle\rangle =\lambda\cdot a_{-o,0}$ for a certain $\lambda\in\C^*$.
Therefore,
the right hand side of (\ref{eq:Tomega11}) tests exactly the non--regular part  of 
 ${\rm Res}_{\widetilde{D}_1}(\omega)$.

\bekezdes\label{ss:TA}  {\bf The sheaf $\Omega_{\tX}^2(Z)^{{\rm regRes}_{\widetilde{D}}}$.}
 Consider again $l'\in \calS'$ and a divisor
  $D \in \eca^{-l'}(Z) $, which is a union of $-(l', E)$ disjoint divisors  $\{D_i\}_i $,
 each of them $\cO_Z$--reduction of divisors $\{\widetilde{D}_i\}_i $ from  $\eca^{-l'}(\tX)$
 intersecting  $E$  transversally.
Set  $\widetilde{D}=\cup_i\widetilde{D}_i $   
 and write $Z=\sum_vm_vE_v$.

We introduce a subsheaf $\Omega_{\tX}^2(Z)^{{\rm regRes}_{\widetilde{D}}}$
of  $\Omega_{\tX}^2(Z)$ consisting of those forms $\omega$,
which have the property that for every  $i$ the residue ${\rm Res}_{\widetilde{D}_i}(\omega)$
has no pole at $\widetilde{D}_i\cap E$.
For more see \cite[10.1]{NNI}.

\begin{theorem}\label{th:Formsres} \ \cite[Th. 10.1.1]{NNI}
In  the above situation  one has the following facts.

(a) The sheaves $\Omega_{\tX}^2(Z)^{{\rm regRes}_{\widetilde{D}}}/\Omega_{\tX}^2$ and $\calO_{Z}(K_{\tX}+Z-D)$ are isomorphic.

(b) $H^0(\tX,\Omega_{\tX}^2(Z)^{{\rm regRes}_{\widetilde{D}}})/
H^0(\tX,\Omega_{\tX}^2)\simeq H^0(Z,\calO_Z(K_{\tX}+Z-D))\simeq H^1(Z,\calO_Z(D))^*$.

\end{theorem}

\section{The possible $p_g$ and $h^1(\calL)$ values}\label{s:h^1values}

\subsection{`Subsingularities'}\label{ss:subsing}
In the  next discussions  it is convenient to use the following terminology. Let
$\phi: \tX \to X$ be a resolution with dual graph $\Gamma$.
Assume that $\Gamma_r$ is a full connected subgraph of $\Gamma$ with vertices $\calv_r$.
Let $\tX_r$ be a convenient small tubular
neighbourhood of  $E_r:= \cup_{v\in\calv_r}E_v$. By Grauert theorem \cite{GRa} $E_r$ can be contracted in
$\tX_r$, which give rise to a normal singularity $(X_r,o_r)$. We will call
$(X_r,o_r)$   a `subsingularity'  of $(X,o)$ with respect to $\phi$ and $\Gamma_r$. Or, just we say simply that
$\tX_r$ is a subsingularity of $\tX$ associated with $\Gamma_r$.

\subsection{}

 The next lemma will be useful in the next  proofs. For a slightly different version see
\cite{Nagy-holes}.

\begin{lemma}\label{lem:1} Consider $\phi:\tX\to X$  as above with cohomological cycle $Z_{coh}=\sum_v m_vE_v$.
Assume that  $m_u\geq 2$ for some $E_u\subset |Z_{coh}|$.
We blow up $E_u$ sequentially along generic points $m_u - 1$ times (this means that we blow up $E_u$ at a generic point,
say at  $p$, then we blow up the newly created exceptional curve at a generic point, etc.).
Let the last newly created  curve be  $E_{u'}$.
Let us denote the new modification  by $\pi:\tX_{new}\to \tX$, the new resolution by
$\tX_{new}\to X$, whose vertex set will be denoted by $\calv_{new}$. Let $\tX_{new}'$ denote a  convenient
small neighbourhood of the exceptional divisors indexed by $\calv_{new}\setminus \{u'\}$. Then $h^1(\calO_{\tX_{new}'})=h^1(\calO_{\tX_{new}})-1$.
This means that if $(X',o)$ denotes the singularity obtained from $\tX_{new}'$ by contracting its exceptional divisors, then  $p_g(X',o) = p_g(X,o) - 1$.

(Note that $u'$ is an end vertex, hence
the dual graph of $\tX_{new}'$ is connected.)
\end{lemma}

\begin{proof}
%
Let $Z_{coh}=Z_{coh}(\tX)$ (resp.  $Z_{coh}(\tX_{new})$)
denote the cohomological cycle of $\tX$ (resp. $\tX_{new}$).


First   we claim  that $p_g(X',o) < p_g(X,o)$.

Indeed, by Lemma \ref{lem:blowupcoh} the cohomological cycle of $\tX_{coh}$ is not supported in $\tX_{coh}'$, hence,
by the definition of the cohomological cycle $p_g(X',o)=h^1(\calO_{\tX_{new}'}) < h^1(\calO_{Z_{coh}(\tX_{new})})=p_g(X,o)$.
%

For the opposite inequality, let us compare in the sequence of blow ups the last step $\tX_{new}$ with the previous step
whose resolution space will be denoted by   $\tX_{new}^{pr}$. Let $b:\tX_{new}\to
\tX_{new}^{pr}$ be the blow up.
Let $Z^{pr}$ and $Z$ be the cohomology cycle of $\tX_{new}^{pr}$ and $\tX_{new}$
respectively. Then by Lemma \ref{lem:blowupcoh} $Z=b^*Z^{pr}-E_{u'}$ and the $E_{u'}$--multiplicity of $Z$ is one.
Then one sees that the cohomological cycle of $\tX_{new}'$ is smaller than $Z-E_{u'}$. Then using the surjection
 $\calO_Z\to \calO_{Z-E_{u'}}$ we get that $p_g(X,0)-p_g(X',o)\leq h^1(\calO_{E_{u'}}(2E_{u'}-b^*(Z^{pr}))=1$.
\end{proof}

\begin{remark}\label{rem:GENERALg}
In Lemma \ref{lem:1}
the assumption $m_u\geq 2$ cannot be replaced by $m_u\geq 1$,
see e.g. the cone--like case presented in Example  \ref{rem:3.3.3}.
\end{remark}

\subsection{}\label{conj1}
In the sequel we prove a {\bf conjecture} of the third author formulated in  \cite[Conjecture 2.8]{Okuma19} as follows:
{\it  if $(X, o)$ is a complex normal surface singularity and we fix  an arbitrary integer
$0 \leq q \leq p_g(X)$ then there exists a resolution $\tX \to X$ and an effective integral  cycle $Z > 0$ such that $\calO_{\tX}(-Z)$ is base point free and $h^1(\calO_{\tX}(-Z)) = q$.}

\begin{theorem}\label{t:q(I)}
 Fix a  complex normal surface singularity $(X, o)$ such that all the irreducible exceptional
 divisors are rational (however $\Gamma$ is not necessarily a tree).
 Fix also an  arbitrary integer $0 \leq q \leq p_g(X,o)$.
Then the following facts hold.

(1) There exists a resolution $\tX \to X$, and a  subsingularity $X_r$ with resolution
$\tX_r\subset \tX$ associated with a
 connected full subgraph $\Gamma_r \subset \Gamma$ (cf. \ref{ss:subsing}) such that
$p_g(X_r,o_r) = q$.

(2)  There exists a resolution $\tX \to X$,   which admits a   $q$--cohomological cycle.

(3) There exists a resolution $\tX \to X$ and an effective integral cycle $l > 0$ such that $\calO_{\tX}(-l)$ is base point free and $h^1(\calO_{\tX}(-l)) = q$.
\end{theorem}
\begin{proof}

We prove part {\it (1)} by a decreasing  induction on $q$.
 If $q = p_g(X,o)$, then the statement is trivial, since we can take any resolution $\tX \to X$, $\tX_r=\tX$ and
 $\Gamma_r=\Gamma$.

Next, assume that for a certain $0<q\leq p_g(X,o)$
we already know the validity of  the statement, that is,
we have a resolution $\tX \to X$, and a subresolution/subsingularity $\tX_r$  associated with the
subgraph $\Gamma_r \subset \Gamma$,
 such that $p_g(X_r,o) = q$.

Let us denote the cohomological cycle of $\tX_r$ by $Z_r=\sum_{v\in \calv(\Gamma_r)}m_vE_v$.

First, assume that
 there exists a vertex $u \in \calv(\Gamma_r)$ such that $m_u> 1$.
 Then we perform  the construction of  Lemma \ref{lem:1}:
 we blow up  $E_u$ sequentially along generic points $m_u- 1$ times and let the last created  curve be $E_{u'}$.
Let us  denote the new vertex set by $\calv_{new}$, and the new resolution by $\tX_{new}$.
Furthermore, we have its subsingularity $\tX_{r, new}$ with vertex set
$\calv_{r, new}$, where $\tX_{r, new}$ is the modification of $\tX_{r}$  by the previous sequence of blow ups.
Clearly, $h^1(\calO_{\tX_{r, new}}) = p_g(X_{r}) = q$.
Additionally, let us denote the subresolution  of $\tX_{r, new}$ corresponding to the vertex set $\calv_{r, new} \setminus \{u'\}$ by $\tX_{r, new}'$. Its dual graph is connected and
 by Lemma \ref{lem:1}
 $h^1(\calO_{\tX_{r,new}'}) =h^1(\calO_{\tX_{r,new}}) -1$,  which finishes the induction step and the proof of part {\it (1)}
 in this case.

Second, assume that $m_v\leq 1$ for all $v\in \calv(\Gamma_r)$. Then  $p_g(\tX_r) $ equals the number of independent 1--cycles in
$\Gamma_r$. Since  $p_g(\tX_r)>0$, we necessarily have at least one such cycle in $\Gamma_r$. Let us fix an edge $e$ of
$\Gamma_r$ such that the graph obtained from $\Gamma_r$ by deleting of $e$ has one less independent 1--cycles.
Then, if we blow up $e$ end we delete the newly created irreducible exceptional divisor with its adjacent edges we get a full
connected subgraph with the required property .

This ends the proof of part {\it (1)}.

For part {\it (2)},
 take a resolution $\tX$ and $\tX_r\subset \tX$ with $h^1(\tX_r)=q$ as in {\it (1)}.
Then the cohomological cycle of $\tX_r$ is  a $q$--cohomological  cycle of $\tX$.

Finally, part {\it (3)} follows  from {\it (2)} via
 the following  lemma (with slightly stronger statement).
\end{proof}
\begin{lem}\label{l:q-cycle}
If there exists a $q$-cohomological cycle on $\tX$, then there exists a resolution $\tX_{new}\to X$, which factors through $\tX$, and a cycle $l>0$ on $\tX_{new}$ such that $\cO_{\tX_{new}}(-l)$ is base point free and $h^1(\cO_{\tX_{new}}(-l))=q$.
\end{lem}
\begin{proof}
Let $C$ be a $q$-cohomological cycle of $\tX$, and write $C= \sum_{E_u\subset |C|}m_uE_u$.

Next, we  fix
 a cycle $l$ such that $(l,E_v)<0$ for all $E_v$ and $\cO_{\tX}(-l)$ is base point free.
 Let $f\in H^0(\cO_{\tX}(-l))$ be a general element and write $\di_{\tX}(f)=l+D$,
 where $D$ has no exceptional components.
 We write $D$ as  $D_1+D_2$ with disjoint union supports, such that all the components $\{D_{1,i}\}_i$
 of $D_1$ intersect
 $|C|$, while $D_2\cap |C|=\emptyset$. Since $\cO_{\tX}(-l)$ is base point free, we can chose $f$ such that
 the intersection points $D_1\cap |C|$ are generic with respect to the sections
 of $H^0(\tX, \calO_{\tX}(K_{\tX}+C))$. 
 Assume that  $p_i:=D_{1,i}\cap |C|\in E_{u_i}$. For each $i$,  we blow up  $\tX$ at $m_{u_i}$
 infinitely near point at $p_i$,
 let the created exceptional curves be $\{F_{i, j}\}_{1\leq j\leq m_{u_i}}$ ($F_{i,1}$ being the very first one).
 Denote the strict transforms of $E_v$
 by the same symbol $E_v$ ($v\in\calv$), the strict transform of $D_1$ by $\widetilde{D}_1$.
 Let $\tX_{new} $ be the new resolution, set $\pi:\tX_{new}\to \tX$.

  Then by the fact that the points $p_i$ are not base points of $\calO_{\tX}(K_{\tX}+C)$, we obtain that the support $S=\cup_{v}E_v\cup \cup_{i, 1\leq j< m_{u_i}}F_{i,j}$ has the following properties
   (cf. Lemma \ref{lem:blowupcoh} and its proof)

 (1) $S$ consists of those exceptional curves of $\tX_{new}$ which do not intersect $\widetilde{D}_1\cup D_2$.

 (2) The cohomological cycle $C'$ on  $S$ is $\pi^*(C)-\sum_{i,1\leq j<m_{u_i}}jF_{i,j}$, and its support is $S$.

 Finally, write $\di_{\tX_{new}}(f\circ \pi)$ as $l_{new}+\widetilde{D}_1+D_2$.
 Then $\calO_{\tX_{new}}(-l_{new})$ has no fixed components, hence {\it (1)} and {\it (2)} together with
  \cite[3.6]{O17} imply that for $n\gg 0$ the bundle  $\cO_{\tX_{new}}(-n\cdot l_{new})$ is base point free and
 $h^1(\cO_{\tX_{new}}(-n\cdot l_{new}))=h^1(\calO_{C'})$, cf. (\ref{eq:free}). But $h^1(\calO_{C'})=h^1(\calO_C)=q$.
\end{proof}

\begin{example}\label{rem:3.3.3} (a)  Consider
a cone--like singularity whose minimal resolution  is a smooth curve $C$ of genus $g\geq 2$ with self--intersection
very negative (compared with $g$).  Then, using the exact sequence
$0\to \calO_{C}(-mC)\to \calO_{(m+1)C}\to \calO_{mC}\to 0$ and the vanishing
$h^1(\calO_C(-mC))=0$ (for all $m\geq 1$)
we have $p_g(X,o)=h^1(C,\calO_C)=g$.
Moreover, for an  arbitrary resolution $\tX$, if we denote the strict transform of
$C$ by the same symbol $C$,  then $h^1(\tX,\calO_{\tX})=p_g(X,o)=h^1(C,\calO_C)=g$ too.
We claim that part {\it (1)} of Theorem \ref{t:q(I)} does not hold for any
$0<k<p_g(X,o)$. Indeed, take an arbitrary  resolution $\tX$ and $\tX_r\subset \tX$ as in {\it (1)}.
Then, if the strict transform of $C$ is in $\Gamma_r$ then $h^1(\calO_{\tX_r})=g$, otherwise it is zero.

(b) Let us start with a cone--like  singularity as in part (a).  Let $\tX$ be any resolution and fix
a line bundle $\calL\in \pic(\tX)$ without fixed components.
We claim that $h^1(\tX, \calL)=h^1(C,\calL|_C)$. Indeed, by the exact sequence
$0\to \calO_{\tX}(-C)\to \calO_{\tX}\to \calO_C\to 0$ (and part (a)) we get that
$h^1(\calO_{\tX}(-C))=0$. On the other hand, by multiplication with a generic global section of
$\calL$ we have an exact sequence $0\to \calO_{\tX}(-C)\to \calL(-C)\to A\to 0$, such that the support of
 $A$ is Stein.
Hence $H^1( \calO_{\tX}(-C))\to H^1( \calL(-C))$ is onto, hence $H^1(\tX,\calL(-C))=0$ too. This proves the claim.

Let us analyse the validity of part {\it (2)} of Theorem \ref{t:q(I)} in this case. Fix any resolution $\tX$  and
a base point free line bundle $\calL\in {\rm Pic}(\tX)$. If the degree $d$ of
the restriction $\calL|_C\in \pic(C)$  is zero then
$h^1(\tX,\calL)=g$. Otherwise  the degree is necessarily $d\geq 2$. Assume next this case.
%
%
%
%
Next we analyse the line bundle $\calL|_C\in \pic(C)$ of degree $d\geq 2$.

If $d>2g-2$ then $h^1(C,\calL|_C)=0$. If $d=2g-2$ then $h^1(C,\calL|_C)=h^0(C, \calL^{-1}|_C\otimes K_C)$ is 0 or 1.
If  $d<2g-2$ then by Clifford theorem (and Riemann--Roch) $h^1(C,\calL|_C)\leq g-d/2$ and if the equality holds then $C$ should be  hyperelliptic.
In particular,  if $C$ is not hyperelliptic, and $g\geq 3$ then
$h^1(\tX,\calL)=g-1$ cannot be realized.

If the degree $d$ is larger (than 2) then we get even a larger gap for $h^1(\tX,\calL)$.
\end{example}

\begin{remark}\label{rem:GEN} In Theorem \ref{t:q(I)} in fact we proved the following fact.
 Fix a  complex normal surface singularity $(X, o)$ and
  an  arbitrary integer $q$ such that $\sum_vg_v \leq q \leq p_g(X,o)$.
Then there exists a resolution $\tX \to X$ and a  subsingularity $X_r$ with resolution
$\tX_r\subset \tX$ associated with a
 connected full subgraph $\Gamma_r \subset \Gamma$ such that
$p_g(X_r,o_r) = q$.

Indeed, in this case the induction runs as follows.
If $Z_{coh}$ has a  coefficient with $m_u>1$ then we proceed as in the above proof and we can find a subsingularity with
geometric genus one less. If $Z_{coh}$ is reduced then $p_g=\sum_vg_v+c_{\Gamma}$, where
$c_{\Gamma}$ is the number of independent cycles in $\Gamma$. But the number of such cycles can  also
be decreased one by one (by deleting the exceptional curve of a blow up at a conveniently chosen  singular point of $E$).
\end{remark}

\section{Stability bound for  $n\mapsto h^1(\tX, \calL^{n})$.}\label{s:STAB}

\subsection{} Before we state the next result we wish to make the following preparation.

\bekezdes Let $\tX$ be a resolution and $Z\in L_{>0}$ an effective cycle  and
$\calL\in \pic^{l'}(Z)$ a line bundle without fixed components.
Fix a generic section $s\in H^0(Z,\calL)$. Then the cohomology long exact sequence
of $0\to\calL^{ n}\stackrel{\times s}{\longrightarrow} \calL^{ n+1}\to A\to 0$ (where $\times s$ is
the mulitplication by $s$ and the support of $A$ is zero-dimensional)  shows that
$h^1(Z,\calL^{n})\geq h^1(Z,\calL^{n+1})$. Hence the sequence $n\mapsto h^1(Z,\calL^{ n})$
is non--increasing  and it is constant for $n\gg 0$.

\begin{definition}
For any line bundle $\calL\in {\rm Pic}^{l'}(Z)$ without fixed components let $n_0(Z,\calL)$ be the smallest
integer $n$ such that $h^1(Z,\calL^{n})= h^1(Z,\calL^{n+1})$.
\end{definition}
\begin{remark}\label{rem:stab0}
Using the relevant exact sequences one  verifies that  $n_0(Z,\calL)=
\min\{n\ge 0 \,:\, h^1(Z,\calL^n)=h^1(Z,\calL^{m}), \forall m>n\}$, see e.g. \cite[Lemma 3.6]{O17}.
\end{remark}
\bekezdes Assume that the {\it link is a rational homology sphere}, hence the theory of Abel maps can be applied.
Similarly as above, if we fix a Chern class $l'\in-\calS'$, then $n\mapsto \dim(\im(c^{nl'}(Z)))$ is a non--decreasing sequence
bounded from above by $h^1(\calO_Z)$, hence it must stabilise, cf. \ref{bek:I}.
The limit  $e_Z(l')$  depends  only on the $E^*$--support $I=I(l')$ of $l'$, hence will also be denote by $e_Z(I)$.

\begin{definition}
For any Chern class  $l'\in - \calS'$  let $n_0'(Z,l')$ be the smallest
integer $n$ such that  $\dim(\im(c^{nl'}(Z)))= \dim(\im(c^{(n+1)l'}(Z)))$.
\end{definition}
\begin{remark}\label{rem:stab} A similar type of stability is valid for $n_0'$ as in Remark \ref{rem:stab0}. Namely,
$n_0'(Z,l')=
\min\{ n\ge 0\,:\,
\dim(\im(c^{nl'}(Z)))= \dim(\im(c^{ml'}(Z))), \forall m>n\}$. This follows from \cite[\S 6]{NNI}:
 $ \dim(\im(c^{nl'}(Z)))= \dim(\im(c^{(n+1)l'}(Z)))$ holds if and only if
 $\dim(\im(c^{nl'}(Z)))$ equals the dimension of the affine closure of $\im(c^{l'}(Z))$.
\end{remark}
\bekezdes
One has the following relations between the `stabilized dimensions'.
By \cite[Remark 3.8]{O17} or \cite[Th. 6.1.9]{NNI} we know that for $\calL$ without fixed components
$\lim_{n\to \infty} h^1(Z,\calL^{ n})=h^1(\calO_{Z|_{\calv\setminus I(l')}})$, and by (\ref{eq:ezl})
$e_Z(I)=h^1(\calO_Z)-h^1(\calO_{Z|_{\calv\setminus I(l')}})$.
In other words, whenever $\calL\in \im (c^{l'}(Z))$, we have
$$\lim_{n\to \infty} h^1(Z,\calL^{ n})=\lim_{n\to \infty} {\rm codim}(\im(c^{nl'}(Z))).$$
Furthermore, we have the following geometric interpretations in terms of differential forms as well.
By \cite[(8.3.1)]{NNI}, see also (\ref{eq:ezlc}),
$h^1(\calO_{Z|_{\calv\setminus I(l')}})=h^1(\calO_Z)-e_Z(I)$
 equals $\dim \Omega_Z(I)$, where $\Omega_{Z}(I)$ is
 that  subspace  $H^0(\tX, \Omega^2_{\tX}(Z|_{\calv\setminus I}))/
H^0(\tX, \Omega^2_{\tX})$ in  $H^0(\tX, \Omega^2_{\tX}(Z))/
H^0(\tX, \Omega^2_{\tX})$, which is
generated by forms which have no poles along (the generic points of) $\{E_u\}_{u\in I}$.
Furthermore, by Theorem \ref{th:Formsres}),
the dimension of a cohomology group $H^1(Z, \calL)$
can be determined via forms  whose  Leray residues have no pole along $D$, where
$D$ is a  transversal sections as in \ref{ss:TA}.

This is complemented with the following comparison.

\begin{lemma}\label{lem:COMP}  Assume that for certain integer $n_0$ there exists
 a line bundle $\calL\in \pic^{l'}(Z)$ without fixed components such that
$n\mapsto h^1(Z, \calL^{ n})$ is constant for any $n\geq n_0$.
Then
 $n\mapsto \dim(\im(c^{nl'}(Z)))$ is constant for $n\geq n_0$ as well. In other words, for any $l'\in-\calS'$,
 $$n_0'(Z,l')\leq \min\,\{n_0(Z,\calL)\, :\, \calL\in \im (c^{l'}(Z))\}.$$
\end{lemma}
\begin{proof}
By assumption and the above discussion, for any $n\geq n_0$ one has
$h^1(\calO_Z)-e_Z(I)= h^1(Z, \calL^{ n})$. Moreover, by (\ref{eq:dimfiber2}),
$h^1(Z, \calL^{ n})\geq h^1(\calO_Z)-\dim(\im(c^{nl'}(Z)))$, hence
$\dim(\im(c^{nl'}(Z)))\geq  e_Z(I)$. But  $n\mapsto \dim(\im(c^{nl'}(Z)))$
is non--decreasing with limit $e_Z(I)$, hence $\dim(\im(c^{nl'}(Z)))= e_Z(I)$.
\end{proof}
\subsection{}
In this subsection we prove positively a {\bf question} of the third author formulated as
Problem 3.13 in  \cite{Okuma19} as follows:
{\it if $\tX \to X$ is an arbitrary resolution and $l_0\in L_{>0}$ is an effective integral
cycle such that the line bundle
$\calO_{\tX}(-l_0)$ is base point free, then $n\mapsto h^1(\calO_{\tX}(-n\cdot l_0))$ is constant  for
$n \geq 1 - \min_{l \geq E} \chi(l)$.}

Note that $\min_{l\geq E}\chi(l)=\min_{l>0}\chi(l)$.
In the literature the integer $\min_{l>0}\chi(l)$ was already considered in rather different situations.
 In \cite{Wa70}
Wagreich called $1-\min_{l>0}\chi(l)$ {\it arithmetical genus} $p_a$
of $(X,o)$,  and for any non--rational germ he proved that
$p_a\leq p_g$ (see \cite[p. 425]{Wa70}).
Furthermore, the bound
 $ 1 - \min_{l >0} \chi(l)$  has the following remarkable appearance  too under
 the assumption that the link is a rational homology sphere:
 it  is the smallest possible
 geometric genus of any singularity (analytic type) with the topological type fixed by the dual
 graph $\Gamma$ of $\tX$, see \cite{NNII}. Similarly, for any $Z\geq E$, the integer
$ 1 - \min_{Z\geq l>0} \chi(l)$ is the smallest possible value of
 $h^1(\calO_Z)$ associated with any analytic  singularity type
 with the same topological type fixed by the dual
 graph $\Gamma$ see \cite{NNII}. Both are realized by the generic analytic structure.

\subsection{} In fact,
we prove the statement for any base point free line bundle and any $Z\geq E$.


\begin{theorem}\label{th:STAB}
Assume that $(X,o)$ is  a normal surface singularity whose link is a rational homology sphere and let
us fix  a resolution $\tX \to X$ and an effective  cycle $Z\in L$, $Z\geq E$. 

Assume that $\calL\in\pic^{l'}(Z)$, $l'\in-\calS'$,  is a base point free line bundle on $Z$,
and write $I:=I(l')=\{v \in \calv \,:\,  (c_1\calL, E_v) \not=0\}$ as above.
Then $h^1(Z,  \calL^{n}) =  h^1(\calO_Z) - e_Z(I)$ whenever  $n \geq 1 - \min_{Z\geq l>0} \chi(l)$.
\end{theorem}

\begin{proof} If $\calL\in \pic^{l'}(Z)$ is base point free, then
the divisor of  a generic section  of $\calL$ is the restriction of a divisor $\widetilde{D}$ of $\tX$,
 which intersects $E$ transversally (for $Z\gg 0$ see
e.g. \cite[Remark 9.1.3]{NNIV}, the general case follows similarly).
Note that $\widetilde{D}$ usually has many components.


Assume that for a certain $n\geq 2$ with
\begin{equation} \label{A}
 n\geq 1-\min_{Z\geq l>0} \chi(l) \end{equation}
one has $h^1(Z,  \calL^{n}) >  h^1(\calO_Z) - e_Z(I)$.
Recall that $h^1(\calO_Z)-e_Z(I)=\dim \Omega_Z(I)$,
 is the dimension of that subspace of
 $H^0(\tX, \Omega^2_{\tX}(Z))/
H^0(\tX, \Omega^2_{\tX})$, which is
generated by forms which have no poles along $\{E_u\}_{u\in I}$, cf. (\ref{eq:ezlc}) and (\ref{eq:ezl}).
Also, $h^1(Z,\calL^{ n})$ is the dimension of the subspace of classes of forms
whose  Leray residues
 along a set of fixed transversal sections of $\calL^{n}$ have no poles,  cf. Theorem \ref{th:Formsres}.

 Consider any set of generic
sections $s_1, s_2, \ldots, s_n \in H^0(Z, \calL)_{\reg}$ such that for any $i$ the divisor of $s_i$
is the restriction of a divisor $D_i$ of $\tX$, and each $D_i$ intersects $E$ transversally, and
 $D_i\cap D_j=\emptyset$ for $i\not= j$. Then, by the above numerical identifications, the inequality
 $h^1(Z,  \calL^{n}) >  h^1(\calO_Z) - e_Z(I)$ can happen only if
there exists a differential form $\omega \in H^0(\tX, \Omega^2_{\tX}(Z))/H^0(\tX, \Omega^2_{\tX})$
such that the Leray residues along each component of $\cup_i D_i$ has no pole,  but $\omega$
has a non--trivial  pole along a certain $E_u$  with $u \in I$ (cf. the previous paragraph).

Since the set of poles of relevant forms  runs over a finite set (hence for a generic choice of
the sections $\{s_i\}_i$ the pole of $\omega$ is stable)
 we obtain that there exists a fixed non--zero effective integral cycle
 $Z'\leq Z$ with $Z'\geq E_u$ for a certain $u\in I$,
 such that for any generic choice of $s_1,  \ldots, s_n \in H^0(Z, \calL)_\reg$
 (with transversal sections $D_i$ as above)
 we can find a differential form $\omega$ with pole  $Z'$ which  has regular  Leray residue along
 each component of  $\cup_iD_i$.

Let us fix such such a cycle $Z'$ and  $E_u$.
Write $m_u=-(E_u^*, Z)$ for the $E_u$--multiplicity of $Z$.

Since $(E_u,c_1\calL)\not=0$,  for any $i$  one of the components of $D_i$  intersects $E_u$.
Hence,  we can choose a point
$p_i\in E_u$ from the support $|D_i|$ of $D_i$.
Since $\calL$ is
base point free, this intersection point is even generic on $E_u$.

 Fix some local coordinates $(x,y)$ of the germ $(\tX,p_i)$  with
$\{x=0\}=E_u$, $\{y=0\}=D_{p_i} $,  the corresponding component $D_{p_i}$ of $D_i$.
Write $\omega$ in coordinates $(x,y)$,
 $\omega=\varphi(x,y)dx\wedge dy/x^o$ (modulo $x^{m_u}$), where
$\varphi(x,y)$ is locally holomorphic, and $o>0$ is the pole order along $E_u$, $x\nmid\varphi$.
Since the Leray residue $\varphi|_{y=0}dx/x^o$ has no pole along this $x$-axis, necessarily
$\varphi$ has the form $x^o\alpha(x)+y\beta(x,y)$, hence
 $\varphi(0,0)=0$. In  such a case we say that `$\omega$ vanishes at $p_i$'.

Since $\calL$ is base point free, these points $p_i$ can be chosen freely on $E_u$.
(Note that the point $p_i$ determines and is determined by $D_{p_i}$. Hence, generic $s_i$ provides a generic point
$p_i$. Furthermore, the choices of $s_i$ and $s_j$ for $i\not =j$ --- i.e. the choices of $p_i$ and $p_j$ ---
 are independent.)
Therefore,
we obtain that  for generic points $p_1, \ldots, p_n$ of $E_u$
 there exists a certain $\omega$ as above, which   vanishes at $p_1,\ldots, p_n$.
In other words,
for generic points $p_1, \ldots, p_n \in E_u$, there is a section
$s \in H^0(\tX, \Omega^2_{\tX}(Z'))\setminus
H^0(\tX, \Omega^2_{\tX}(Z'-E_u))$, which vanishes  at the points $p_1, \ldots, p_n$.

Consider next the exact sequence
\begin{equation}
0 \to H^0(\tX, \Omega^2_{\tX}(Z'-E_u))  \to   H^0(\tX, \Omega^2_{\tX}(Z')) \stackrel {\alpha}{\longrightarrow}
 H^0(E_u, \Omega^2_{\tX}(Z'))
\end{equation}
and let $V:=\im (\alpha)$.
It is the vector space
of a  linear system on $E_u$ such that for any generic $p_1, \ldots, p_n \in E_u$, there exists $s\in V$, $s\not=0$,
which vanishes at all the points $p_1,\ldots , p_n$. In particular,  $\dim(V)\geq  n+1$.
This combined with the assumption (\ref{A}) we get $\dim(V)\geq   2 - \min_{Z\geq l>0} \chi(l)$.

Hence,
we get that $\dim(\, H^0(\tX, \Omega^2_{\tX}(Z'))/
 H^0(\tX, \Omega^2_{\tX}(Z'-E_u))\,)\geq  2 - \min_{Z\geq l>0} \chi(l)$.

Since $\dim  H^0(\tX, \Omega^2_{\tX}(Z'))/H^0(\tX, \Omega^2_{\tX})=h^1(\calO_{Z'})$ (see e.g. \ref{ss:LauferD} or \cite[(7.1.40]{NNI}), and similarly
 $\dim  H^0(\tX, \Omega^2_{\tX}(Z'-E_u))/H^0(\tX, \Omega^2_{\tX})=h^1(\calO_{Z'-E_u})$
we obtain that
 \begin{equation}\label{eq:uj}
  h^1(\calO_{Z'}) -  h^1(\calO_{Z'- E_u}) \geq  2 - \min_{Z\geq l>0} \chi(l).\end{equation}
On the other hand,
from the following exact sequence
\begin{equation*}
0 \to H^0(\calO_{Z'- E_u}( - E_u)) \to  H^0(\calO_{Z'}) \stackrel{r}{\longrightarrow}
 H^0(\calO_{E_u})\to
H^1(\calO_{Z'- E_u}( - E_u)) \to H^1(\calO_{Z'}) \to  0
\end{equation*}
and from the fact that $r$ is onto, we get  that
 \begin{equation}\label{eq:uj2}
h^1(\calO_{Z'}) = h^1(\calO_{Z'- E_u}( - E_u)).
\end{equation}
Furthermore,
let us  denote the fixed component cycle of the line bundle
$\calO_{Z' - E_u}(-E_u)$ by $0 \leq A \leq Z' - E_u$. Then
$H^0(\calO_{Z' - E_u}(-E_u)) = H^0(\calO_{Z' - E_u- A}(-E_u- A))$ and $H^0(\calO_{Z' - E_u- A}(-E_u- A))_{\reg} \neq \emptyset$. Moreover, from the cohomological long exact sequence of
$$0\to \calO_{Z' - E_u-A}(-E_u-A)\to \calO_{Z' - E_u}(-E_u)\to \calO_{A}(-E_u)\to 0$$ we get
 \begin{equation}\label{eq:uj3}
  h^1(\calO_{Z'- E_u}( - E_u)) =  h^1(\calO_{Z'- E_u- A}( - E_u- A)) + 1 - \chi(E_u + A).
\end{equation}
Finally, if ${\mathcal G} \in \pic(\tilde{Z})$ is a line bundle without fixed components then
$h^1(\tilde{Z},{\mathcal G})\leq h^1(\calO_{\tilde{Z}})$
(see e.g. \cite[Proposition 5.7.1(b)]{NNI}),  hence
 $ h^1(\calO_{Z'- E_u- A}( - E_u- A)) \leq h^1(\calO_{Z'- E_u- A})$. But we also have
 $ h^1(\calO_{Z'- E_u- A}) \leq h^1(\calO_{Z'- E_u})$.
 This combined with (\ref{eq:uj2}) and (\ref{eq:uj3}) gives
   \begin{equation}\label{eq:uj4}
  h^1(\calO_{Z'}) \leq   h^1(\calO_{Z'- E_u}) + 1 - \chi(E_u + A),
\end{equation}
 which contradicts (\ref{eq:uj}).
\end{proof}

\section{Stability bound for  $n\mapsto \dim(\im(c^{nl'}(Z)))$.}\label{s:STAB2}

\subsection{} The next theorem establishes the analogue of Theorem \ref{th:STAB} for
 $n\mapsto \dim (\im (c^{nl'}))$.

\begin{theorem}\label{th:STAB2}
We fix an arbitrary singularity $(X,o)$, a resolution $\tX$, a Chern class $l' \in - S'$ with $E^*$--support $I=I(l')$, and a
 cycle $Z\in L$, $Z\geq E$.
Then the dimension of the image of the Abel map $n\mapsto
\dim(\im(c^{nl'}(Z)))$ is the  constant  $e_Z(I)$ whenever   $ n \geq 1 - \min_{Z\geq l>0} \chi(l)$.
\end{theorem}

Note that if there exists a base point free line bundle $\calL\in \pic^{l'}(Z)$ then Theorem \ref{th:STAB2} follows from
Lemma \ref{lem:COMP} and Theorem \ref{th:STAB}. The proof of the general case is based on different arguments.

Regarding the restriction $Z\geq E$ note that
the above theorem remains true for any $Z>0$ as well. Indeed,  in this general case
we apply the above version restricted to the connected subgraphs of the support $|Z|$ of $Z$.

\begin{proof}
If we fix some $\calL_0 \in \pic^{l'}(Z)$, then each $\pic^{nl'}(Z)$ can be identified with the linear space
$\pic^0(Z)=H^1(\calO_Z)$ via $\calL\mapsto \calL\otimes \calL_0^{-n}$, and the Abel map $c^{nl'}(Z):
\eca^{nl'}(Z)\to \pic^{nl'}(Z)$ with $\widetilde{c}^{nl'}=\widetilde {c}^{nl'}(Z):\eca^{nl'}(Z)\to H^1(\calO_Z)$. (In \cite{NNI} $\calL_0$ is the `natural line bundle' associated with $l'$.)

In this way we can  use the vector space structure of $H^1(\calO_Z)$.
In particular,
\begin{equation}\label{eq:SUM}
\im( \widetilde{c}^{nl'})+\im( \widetilde{c}^{ml'})\subset \im( \widetilde{c}^{nl'+ml'})
\subset \overline { \im( \widetilde{c}^{nl'})+\im( \widetilde{c}^{ml'}) },\end{equation}
where $\overline {\cdot}$ denotes the topological closure.
Clearly, $\im (c^{nl'}) $ and $\im (\widetilde{c}^{nl'})$
can also be identified as subspaces.
Let us denote by $A(\im (c^{nl'}))$ the affine closure of $\im(c^{nl'})$.
Note that  up to an affine translation $A(\im (c^{l'}))$
is the same as the affine closure $A(\im (c^{nl'}))$ for any $n$.
Therefore,  the `stabilized limit'
$\{ \im (c^{nl'})\}_{n\to \infty}$ is the affine closure $A(\im (c^{l'}))$ (up to an affine translation).
Hence, the dimension $n\mapsto \dim(\im ( c^{nl'}))$ stabilises
exactly when $\overline{\im(c^{nl'}(Z))} = A(\im(c^{nl'}(Z)))$. (For more see \cite[6.1]{NNI}.)
Furthermore,
using again the analogue of (\ref{eq:SUM}) for arbitrary two Chern classes,
 if $l' = \sum_{v \in \calv} - a_v E_v^*$, then
 $\overline{\im(c^{nl'}(Z))} = \overline{\sum_{v \in \calv} \im(c^{-n a_vE_v^*}(Z))}$  and
$ A(\im(c^{nl'}(Z)) = \sum _{v \in \calv} A(\im(c^{-n a_v E_v^*}(Z))$.

This means that it is  enough to prove, independently for  each   $v \in \calv $,
that $A(\im(c^{-n E_v^*}(Z)) = \overline{ \im(c^{-n E_v^*}(Z)}$
 for any  $n \geq 1 - \min_{Z\geq l>0} \chi(l)$.

Fix some $v\in \calv$ and let $\{\Gamma_i\}_i$ be the connected full subgraphs of $\Gamma\setminus v$. For any
cycle $W\in L(\Gamma)$ of the form
 $W=tE_v+\sum_iW_i$ with $W_i\in L(\Gamma_i)$, we  set $W_-:= \sum _i W_i$, the
restriction of $W$ to $\cup_i\Gamma_i$. Write also $l'=-E^*_v$ and $I=I(l')=\{v\}$.

 Consider the restriction  $r_n:\pic^{nl'}(Z)\to \pic^{0}(Z_-)$, $\calL\mapsto \calL|_{Z_-}$.
 This is the affine projection associated with the vector space projection $H^1(\calO_Z)\to H^1(\calO_{Z_-})$.
Hence the fiber $r_n^{-1}(\calO_{Z_-})$  has dimension $h^1(\calO_{Z})-h^1(\calO_{Z_-})= e_Z(I)$ (cf. (\ref{eq:ezl})).
Note also that $\im (c^{nl'}(Z)) \subset r_n^{-1}(\calO_{Z})$. In particular,
$\dim (\im ( c^{nl'}(Z)))=e_Z(I)$ if and only if $\im ( c^{nl'}(Z)) $ contains  a Zariski open subset
in $r_n^{-1}(\calO_{Z_-})$.
This in \cite{R} is formulated as `the pair $(nl', \calO_{Z_-})$ is relative dominant'. In \cite[Theorem 4.1.6]{R}
the following criterion  is proved  for relative dominance. This is what we will use  in the proof.

For the convenience of the reader we reproduce that part of  \cite[Theorem 4.1.6]{R} what  is needed.

\vspace{1mm}

\noindent {\bf Claim:} {\it If  for certain fixed $n\geq 1$
\begin{equation}\label{eq:CRIT}
\chi(nE_v^*) - h^1(\calO_{Z_-}) < \chi(nE_v^* + l) - h^1(\calO_{(Z-l)_-}(-l)),
\end{equation}
for every cycle $0 < l \leq Z$, then $\dim (\im ( c^{nl'}(Z)))=e_Z(I)=h^1(\calO_{Z})-h^1(\calO_{Z_-})$.}

 \noindent {\it Proof of the Claim.}

 Choose $\calL$ a generic element of $\pic^{nl'}(Z)$ with $\calL|_{Z_-}=\calO_{Z_-}$
 (that is,  generic in $r_n^{-1}(\calO_{Z_-})$).
 By a Chern class computation  (recall that $l'=-E^*_v$)   we obtain that
 (\ref{eq:CRIT}) is equivalent with
 \begin{equation}\label{eq:CRIT2}
 h^1(\calO_{(Z-l)_-}(-l))+ \chi (Z-l, \calL(-l)) <
 h^1(\calO_{Z_-}) +\chi(Z,\calL).
\end{equation}
 Take first $l=Z$. Then (\ref{eq:CRIT2}) implies $\chi(Z,\calL)> -h^1(\calO_{Z_-})$, or
 $h^0(Z,\calL)> h^1(Z,\calL)-h^1(\calO_{Z_-})$. But from the epimorphism of sheaves
 $\calL \to \calL|_{Z_-}$ we  have
 \begin{equation}\label{eq:dag}
   h^1(Z,\calL)\geq h^1(\calO_{Z_-})\end{equation} hence $h^0(Z,\calL)> 0$.

 Next we  discuss  $H^0(Z,\calL)_{\reg}$. If $H^0(Z,\calL)_{\reg}\not=\emptyset$, then the chosen generic element  $\calL$
 of  $r_n^{-1}(\calO_{Z_-})$ is in the image of $c^{nl'}(Z)$, hence
 $\dim(\im (c^{nl'}(Z)))=\dim (r_n^{-1}(\calO_{Z_-}))=e_Z(I)$.

 Next, assume that  $H^0(Z,\calL)_{\reg}=\emptyset$. Then there exists $0<l\leq Z$ (the cycle of fixed components of
 $H^0(Z,\calL)$) such that
 \begin{equation}\label{eq:FIXED}
 H^0(Z-l, \calL(-l))=H^0(Z,\calL) \ \ \ \mbox{and} \ \ \ \  H^0(Z-l, \calL(-l))_{\reg}\not=\emptyset.
 \end{equation}
Consider the diagram

\begin{equation*}  
\begin{picture}(200,40)(130,0)
\put(50,37){\makebox(0,0)[l]{$
\ \ \eca^{nl'-l}(Z-l)\ \ \ \ \ \stackrel{c}{\longrightarrow} \ \ \ \pic^{nl'-l}(Z-l)$}}
\put(50,8){\makebox(0,0)[l]{$
\eca^{nl'-l}((Z-l)_-)\ \ \stackrel{c_-}{\longrightarrow} \  \pic^{nl'-l}((Z-l)_-)$}}
\put(76,22){\makebox(0,0){{\tiny $r_1$}$ \downarrow$}}
\put(192,22){\makebox(0,0){$\downarrow \, $\tiny{$r_2$}}}
\put(250,37) {\makebox(0,0)[l]{$ \ni \calL(-l)|_{Z-l}$}}
\put(250,8) {\makebox(0,0)[l]{$ \ni \calL(-l)|_{(Z-l)_-}=\calO_{(Z-l)_-}(-l)$}}
\end{picture}
\end{equation*}

\noindent
By a local computation on charts of the Cartier divisors we get that
$r_1$ is dominant (in fact, it is a submersion over any point of ${\rm im}(r_1)$),
hence the dimension of the generic fiber over
$(c_-^{-1})(\calO_{(Z-l)_-}(-l))$ is (cf. (\ref{eq:DIM}))
\begin{equation}\label{eq:NEW2}
\dim \eca^{nl'-l}(Z-l)-\eca^{nl'-l}((Z-l)_-)=
(nl'-l, Z-l)-(nl'-l,(Z-l)_-).
\end{equation}
On the other hand, by (\ref{eq:dimfiber}) we also have
\begin{equation}\label{eq:NEW3}
\dim \, (c_-^{-1})(\calO_{(Z-l)_-}(-l))= h^1((Z-l)_-,\calO_{(Z-l)_-}(-l))-h^1(\calO_{(Z-l)_-})+(nl'-l,(Z-l)_-).
\end{equation}

 Next we use the above commutative diagram: $c_-\circ r_1=r_2\circ c$.

 The affine map $r_2$ is associated with the linear projection $H^1(\calO_{Z-l})\to H^1(\calO_{(Z-l)_-})$, hence
\begin{equation}\label{eq:NEW1}
\dim\, r_2^{-1}(\calO_{(Z-l)_-}(-l)) = h^1(\calO_{Z-l})- h^1(\calO_{(Z-l)_-}).
\end{equation}

 Note also that $\calL\mapsto \calL(-l)|_{Z-l}$ corresponds to an affine projection,
  hence if $\calL$ is generic in  $\pic^{nl'}(Z)$ with $\calL|_{Z_-}=\calO_{Z_-}$,
 then $\calL(-l)_{Z-l}$ is generic in
 $ r_2^{-1}(\calO_{(Z-l)_-}(-l))$.
 But by  (\ref{eq:FIXED}) it is also an element of $\im (c^{nl'-l}(Z-l))$.
 Since we know from \cite{R} that $(c_-\circ r_1)^{-1}(\calO_{(Z-l)_-}(-l))$ is irreducible,
  we can compute its dimension two ways following the commutative diagram above.
It means that the dimension of the fiber $c^{-1}(\calL(-l)|_{Z-l})$ can be computed by combination of
 (\ref{eq:NEW2})--(\ref{eq:NEW3})--(\ref{eq:NEW1}), and it is
 \begin{equation}\label{eq:NEW4}
\dim \, c^{-1}(\calL(-l)|_{Z-l})= h^1((Z-l)_-,\calO_{(Z-l)_-}(-l))-h^1(\calO_{Z-l})+(nl'-l,Z-l)
\end{equation}
 Finally, this can be compared with (\ref{eq:dimfiber}) applied for $c=c^{nl'-l}(Z-l)$, which gives
 \begin{equation}\label{eq:NEW5}
\dim \, c^{-1}(\calL(-l)|_{Z-l})= h^1(Z-l,\calL(-l))-h^1(\calO_{Z-l})+(nl'-l,Z-l).
\end{equation}
 Therefore, (\ref{eq:NEW4}) and  (\ref{eq:NEW5}) show that
%
%
%
 \begin{equation}\label{eq:h^1}
  h^1(Z-l, \calL(-l))=h^1((Z-l)_-,\calO_{(Z-l)_-}(-l)).
 \end{equation}
 Now, $h^0(Z,\calL)\stackrel{(\ref{eq:FIXED})}{=} h^0(Z-l,\calL(-l))\stackrel{(\ref{eq:h^1})}{=}
 h^1(\calO_{(Z-l)_-}(-l))+ \chi (Z-l, \calL(-l)) \stackrel{(\ref{eq:CRIT2})}{<}
 h^1(\calO_{Z_-}) +\chi(Z,\calL)$. Hence $h^1(Z,\calL)<h^1(\calO_{Z_-})$, which is certainly false,
cf. (\ref{eq:dag}).

 This ends the proof of the Claim.

 \vspace{2mm}

 Therefore, in order to finish the proof of theorem,
 we have to prove that (\ref{eq:CRIT}) holds whenever $0<l\leq Z$ and  $ n \geq 1 - \min_{Z\geq l>0} \chi(l)$.
 With the notation $l=l_-+tE_v$ this reads as
 \begin{equation}\label{eq:CRIT3}
- h^1(\calO_{Z_-}) < \chi(l_-+t E_v) + t \cdot n - h^1(\calO_{(Z-l)_-}(-l_- - t E_v)).
\end{equation}
Now notice that  $H^0(\calO_{(Z-l)_-}(-l_- - t E_v)) \subset H^0(\calO_{Z_-}(- t E_v))$ and the inclusion is proper
 if $t = 0$ (hence  $l_->0$). Via  a computation this yields
\begin{equation}\label{eq:CRIT4}
\chi( t E_v) - h^1(\calO_{Z_-}(- t E_v)) \leq \chi(l_-+ t E_v) - h^1(\calO_{(Z-l)_-}(-l_- - t E_v)),
\end{equation}
such that  the inequality is strict if $t = 0$. This proves (\ref{eq:CRIT3}) for $t=0$. Moreover, in order to prove
(\ref{eq:CRIT3}) for $t\not =0$ it is enough to verify
\begin{equation}\label{eq:CRIT5}
- h^1(\calO_{Z_-}) < \chi( t E_v) + t \cdot n - h^1(\calO_{Z_-}( - t E_v))
\end{equation}
for any $0<t\leq -(Z,E^*_v)$ (and $n$ as in the assumption)
since (\ref{eq:CRIT4}) and (\ref{eq:CRIT5}) imply (\ref{eq:CRIT3}).

Let $A$
denote the fixed component cycle of $\calO_{Z_-}( - t E_v)$, i.e.
 $0 \leq A \leq Z_-$ is the unique cycle, such that $H^0(\calO_{Z_-}( - t E_v)) = H^0(\calO_{Z_--A}( - t E_v- A))$
and $H^0(\calO_{Z_--A}( - t E_v- A))_{\reg} \neq \emptyset$.

Then $h^1(\calO_{Z_-}( - t E_v)) = h^1(\calO_{Z_--A}( - t E_v- A)) + \chi(t E_v) - \chi(t E_v + A)$.
Moreover, since  the line bundle $\calO_{Z_--A}( - t E_v- A)$ has no
 fixed components,  by \cite[Th. 5.7.1]{NNI} we get
  $ h^1(\calO_{Z_--A}( - t E_v- A)) \leq h^1(\calO_{Z_--A})$. But $h^1(\calO_{Z_--A})\leq h^1(\calO_{Z_-})$ too.
  This means that:
\begin{equation}\label{eq:CRIT7}
h^1(\calO_{Z_-}( - t E_v)) \leq h^1(\calO_{Z_-}) + \chi(t E_v) - \chi(t E_v + A).
\end{equation}
This compared with (\ref{eq:CRIT5}) shows that  we have to prove that for $0<t\leq -(Z,E^*_v)$
%
%
\begin{equation*}
1-  \chi(t E_v + A)\leq   t \cdot \big(1 - \min_{Z\geq l>0} \chi(l)\big),
\end{equation*}
which clearly holds.
\end{proof}

\section{Singularities with `distinct pole property'}\label{s:DPP}

\subsection{}
Let $(X,o)$ be an arbitrary singularity  and let $\phi:\tX\to X$ be a fixed resolution. Fix also
$I\subset \calv$ and $Z\geq E$.

We say that $\phi$
satisfies the `distinct pole property' with respect to $I$
if for any $v\in I$ there exists

(i) a  basis $\{[\omega_1], \ldots, [\omega_h]\}$ of
$H^0(\tX, \Omega^2_{\tX}(Z))/H^0(\tX, \Omega^2_{\tX})$ with representatives $\{\omega_n\}_n$, and

(ii)  a partition $J\cup K$ of $\{1,2,\ldots, h\}$ ($J\cap K=\emptyset$)

\noindent
such that the forms $\{\omega_j\}_{j\in J}$ have no pole along $E_v$, while
the pole orders along $E_v$ of the forms $\{\omega_k\}_{k\in K}$ are  all non-trivial  and different.
(Note that this property is independent of the choice of the representatives  $\{\omega_n\}_n$ in
$H^0(\tX, \Omega^2_{\tX}(Z))$.)

Usually, in several examples, when this property holds,
the forms $\{\omega_j\}_{j}$ are determined by some geometric
property and can be chosen independently of the choice of $I$ and $v$.
(See e.g. the case of elliptic singularities \cite[3.4]{NNIII}.)

The main point in this definition is the following.
If one can find a basis with `distinct pole property' with respect to $\{v\}$, then for a generic
divisor $D$ intersecting $E_v$ transversally, the Leray residue  of
$\omega=\sum_n\lambda _n \omega_n $ ($\lambda_n\in\C$) along $D$
is regular if and only if $\omega$ has no pole  along $E_v$.

In particular, if we fix some $l'\in-\calS'$  with $I=I(l')$, and we can find a basis with  `distinct pole property'
with respect to $I$, then for any base point free line bundle $\calL\in\im (c^{l'}(Z))$ by Theorem
\ref{th:Formsres}, (\ref{eq:ezlc}) and (\ref{eq:ezl})
one has $h^1(Z,\calL)=\dim \,\Omega_Z(I)=h^1(\calO_Z)-e_Z(I)$.
That is, $\lim_{n\to \infty} h^1(Z,\calL^{ n})=h^1(Z,\calL)$.
In particular, by Lemma \ref{lem:COMP} $\dim (\im (c^{l'}(Z)))=e_Z(I)= \lim_{n\to \infty}(
\dim (\im (c^{nl'}(Z))))$.

This applied for $I=\calv$  reads as follows
(see also Corollary 2.2.12 and  Remark 5.1.2 of \cite{NNIII}).

\begin{proposition}\label{prop:dpp}
 Assume that $(X,o)$ is not rational and $Z\geq 0$. Fix a resolution $\phi$ and assume that
$\phi$  satisfies the distinct pole property with respect to $\calv$.
 Then the following properties hold:

(1) for any base point free $\calL\in \pic^{l'}(Z)$ we have $n_0(Z,\calL)=1$;

(2) for any Chern class $l'\in-\calS'\setminus \{0\}$ we have $n_0'(Z,l')=1$;
%
%
\end{proposition}

In the sequel let us assume that $Z\gg 0$.

In parallel to the definitions $\bar{r}(I)=1+n_0(\calO_{\tX}(-l))$ (where $I=H^0(\tX, \calO_{\tX}(-l))$,
$I\calO_{\tX}=\calO_{\tX}(-l)$ for a certain resolution $\tX$ and $l\in\calS'\cap L$)
and $\bar{r}(X,o)=\max \{\bar{r}(I)\,:\, I \ \mbox{as above}\}$ we can define the following objects as well.
For any resolution $\tX$ and base point free line bundle $\calL\in {\rm Pic}(\tX)$ set
$\bar{r}(\tX,\calL)=1+n_0(\calL)$ and $\bar{r}_{free}(X,o)=\max\{ \bar{r}(\tX,\calL)\,:\, \mbox{$\tX$ some resolution and} \
\calL\in \rm{Pic}(\tX)\ \mbox{is base point free}\}$.

From definitions $\bar{r}_{free}(X,o)\geq \bar{r}(X,o)$.

Recall that $\brr(X,o)=1$ if and only if $(X,o)$ is rational, cf.  \cite{OWY15b}.

Note that under the assumption of Proposition \ref{prop:dpp} (formulated for  a fixed resolution)
we cannot deduce automatically that $\bar{r}_{free}(X,o)=2$. The point is that even
if a certain resolution satisfies the
distinct pole property, from this fact does not follow that it is satisfies for any resolution.

A typical situation which might appear is the following.
Assume that $E_v\cap E_w\not=\emptyset$, the pole orders of $\omega_1$ (respectively of $\omega_2$)
along $E_v$ and $E_w$ are 2 and 1 (respectively  1 and 2). Then $\{\omega_1,\omega_2\}$  satisfies the distinct
pole property with respect $\{v,w\}$. On the other hand, if we blow up an intersection point of $E_v\cap E_w$,
then along the new exceptional divisor the pole orders are the same (2 and 2).

However, in the presence of some additional properties of the pole cycles we have the following.

\begin{lemma}\label{lem:GLOB}
Let $\tX_{min}$ be the minimal resolution of $(X,o)$. Assume that in this resolution
$\{\omega_n\}_{n=1}^{p_g}$ satisfies the distinct pole property
with respect to $\calv$. Let $P_n$ be the pole cycle of $\omega_n$.
If any of the following two properties hold

(i) $|P_n|\cap |P_m|=\emptyset$ for any $n\not=m$, or

(ii) $P_1\leq P_2\leq \cdots \leq P_{p_g}$,

\noindent
then the distinct pole property holds for any resolution $\tX$ (with respect to the set of vertices of that resolution).
\end{lemma}
\begin{proof}
Use induction with respect to the number of blow ups needed to obtain   $\tX\to \tX_{min}$.
\end{proof}
In particular,  Proposition \ref{prop:dpp} and Lemma \ref{lem:GLOB}  implies the following.

\begin{corollary}\label{cor:dpp}
Assume that $(X,o)$ is not rational.
 If any resolution $\phi$
satisfies the distinct pole property with respect to $\calv$ then  $\brr_{free}(X,o)=2$. In particular
$\brr(X,o)=2$ too.

This statement applies whenever the assumptions of Lemma \ref{lem:GLOB} hold.
\end{corollary}

\begin{example}
Assume that $(X,o)$ is an elliptic singularity.
\marginpar{ do we need $QHS^3$ link}
In this case $p_g$ might depend on the analytic structure supported on the
elliptic topological type.
However, for the minimal resolution the distinct pole property with respect to $\calv$ together  with property
 (ii) from Lemma \ref{lem:GLOB} are  satisfied
(see e.g. \cite[3.4]{NNIII}).
 In particular  $\brr_{free}(X,o)=\brr(X,o)=2$. This provides a new proof of the identity
 $ \brr(X,o)=2$,  valid  for elliptic germs,  proved originally in \cite{O17}.
\end{example}

\begin{example}
Consider the following minimal good resolution graph:

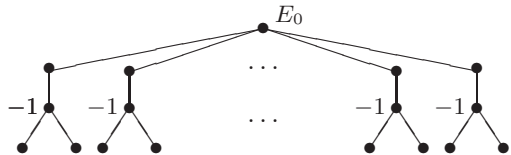
\begin{picture}(200,70)(-50,0)
\put(100,55){\circle*{4}}

\put(20,40){\circle*{4}}\put(50,39){\circle*{4}}\put(150,39){\circle*{4}}\put(180,40){\circle*{4}}
\put(20,25){\circle*{4}}\put(50,25){\circle*{4}}\put(150,25){\circle*{4}}\put(180,25){\circle*{4}}

\put(10,10){\circle*{4}}\put(30,10){\circle*{4}}\put(40,10){\circle*{4}}\put(60,10){\circle*{4}}
\put(190,10){\circle*{4}}\put(170,10){\circle*{4}}\put(160,10){\circle*{4}}\put(140,10){\circle*{4}}

\put(100,55){\line(-3,-1){50}}\put(100,55){\line(-5,-1){80}}
\put(100,55){\line(3,-1){50}}\put(100,55){\line(5,-1){80}}
\put(20,40){\line(0,-1){15}}\put(50,40){\line(0,-1){15}}\put(150,40){\line(0,-1){15}}\put(180,40){\line(0,-1){15}}
\put(20,25){\line(-2,-3){10}}\put(50,25){\line(-2,-3){10}}\put(150,25){\line(-2,-3){10}}\put(180,25){\line(-2,-3){10}}
\put(20,25){\line(2,-3){10}}\put(50,25){\line(2,-3){10}}\put(150,25){\line(2,-3){10}}\put(180,25){\line(2,-3){10}}
\put(100,40){\makebox(0,0){$\ldots$}}\put(100,20){\makebox(0,0){$\ldots$}}
\put(110,60){\makebox(0,0){\small{$E_0$}}}
\put(10,25){\makebox(0,0){\small{$-1$}}}
\put(10,25){\makebox(0,0){\small{$-1$}}}
\put(40,25){\makebox(0,0){\small{$-1$}}}
\put(170,25){\makebox(0,0){\small{$-1$}}}
\put(140,25){\makebox(0,0){\small{$-1$}}}
\end{picture}

\noindent where $E_0$ has $n\geq 2$ adjacent edges, all $g_v=0$, and all the unmarked vertices have self-intersection number
$-N$, where $N$ is very large with respect to $n$. Let us denote the $(-1)$--vertices by $v_1, \ldots  ,v_n$.
Set $F=\sum _{i=1}^n E_{v_i}$.
 A computation shows that $ \lfloor Z_K\rfloor=E+F$. Note also that the Artin minimal cycle is
$Z_{min}=E+2F$, hence $\lfloor Z_K\rfloor\leq Z_{min}$.
Let us fix an arbitrary analytic structure $(X,o)$
supported by the topological type given by this graph, and a resolution $\tX$ with the above dual graph.
 Then $h^1(\calO_{Z_{min}})$ can be computed by Laufer
algorithm \cite{Laufer72}, and it turns out that $p_g=h^1(\calO_{E+F})=h^1(\calO_{Z_{min}})=n$ and the
cohomological cycle $Z_{coh}$  is $E+F-E_0$.

 For any $v_i$  consider the
 minimal star--shaped subgraph whose node is this vertex. Then it determines
  a minimally elliptic graph $\Gamma_i$ and singularity with $p_g=1$ and it admits a unique
 differential form with nontrivial pole. Using this we obtain that there exists a collection of forms
 $\omega_1,\ldots , \omega _n$ on $\tX$ such that the pole cycle of $\omega_i$ is non--trivial and it
 is supported on $\Gamma_i$. In particular,
 they satisfy the distinct pole property together with the additional property (i) of Lemma \ref{lem:GLOB}.
 Therefore  $\brr_{free}(X,o)=\brr(X,o)=2$ (for any analytic structure supported on the above $\Gamma$).

 On the other hand, since $n\geq 2$, the graph is not elliptic: $\chi(Z_{min})=1-n<0$.

This answers negatively  \cite[Problem 3.12]{Okuma19} of the third author (which asked
whether the elliptic singularities are characterized by  the property $\brr(X,o)=2$).
In fact, we proved  that there exists a singularity with $\brr(X,o)=2$ but with arbitrary small $\min\chi$
(or with arbitrary high $h^1(\calO_{Z_{min}})$).
\end{example}

\begin{example}\label{ex:dpp}
One can find non--elliptic singularities with $\brr_{free}(X,o)=2$ even among the Gorenstein germs.
Consider the following resolution graph (cf. \cite[Example 5.1.3]{NNIII}).

\begin{picture}(200,50)(-20,0)
\put(230,40){\makebox(0,0){\small{$-2$}}}
\put(90,40){\makebox(0,0){\small{$-2$}}}
\put(110,40){\makebox(0,0){\small{$-1$}}}
\put(170,30){\circle*{4}}\put(190,30){\circle*{4}}\put(210,30){\circle*{4}}\put(230,30){\circle*{4}}
\put(150,30){\circle*{4}}
\put(210,10){\circle*{4}}
\put(130,40){\makebox(0,0){\small{$-7$}}}
\put(150,40){\makebox(0,0){\small{$-3$}}}
\put(170,40){\makebox(0,0){\small{$-3$}}}\put(190,40){\makebox(0,0){\small{$-7$}}}
\put(210,40){\makebox(0,0){\small{$-1$}}}
\put(200,10){\makebox(0,0){\small{$-3$}}}\put(120,10){\makebox(0,0){\small{$-3$}}}
\put(90,20){\makebox(0,0){\small{$E_1$}}}\put(230,20){\makebox(0,0){\small{$E_2$}}}
\put(90,30){\circle*{4}}
\put(110,30){\circle*{4}}
\put(130,30){\circle*{4}}
\put(110,10){\circle*{4}}
\put(90,30){\line(1,0){140}}\put(110,10){\line(0,1){20}}\put(210,10){\line(0,1){20}}
\end{picture}

 The graph is {\it not} elliptic, $\min\chi=-1$.

It is realized e.g. by the hypersurface singularity with non--degenerate Newton boundary
$\{z^3+x^{13}+y^{13}+x^2y^2=0\}$. This analytic structure has $p_g=5$ and
it is clearly Gorenstein. Let
$\omega $ be the Gorenstein form (with pole $Z_K$). Then the classes of the five forms
$\omega, \, \omega x, \, \omega x^2, \,  \omega y, \, \omega y^2$ constitute a basis of
$H^0(\Omega^2_{\tX}(Z))/H^0(\Omega^2_{\tX})$, and they satisfy the
`distinct  pole property' for $Z\gg 0$
(the verification is left to the reader; the divisor of $x$ is $E_1^*$, while the divisor of $y$ is $E^*_2$).
A verification shows that the distinct pole property survives even if we blow up (several  times)
this  $\tX$.

This example shows that Lemma \ref{lem:GLOB} can be generalized to a more general situation regarding
 the structure of the poles
(a combination of properties (i) and (ii)). (These conditions can be compared with the GCD property from \cite{OkMult}.)
\end{example}


\begin{thebibliography}{30}



\bibitem[A62]{Artin62} Artin, M.:
Some numerical criteria for contractibility of curves on algebraic surfaces.
{\em  Amer. J. of Math.}, {\bf 84}, 485-496, 1962.

\bibitem[A66]{Artin66} Artin, M.:
On isolated rational singularities of surfaces.
{\em Amer. J. of Math.}, {\bf 88}, 129-136, 1966.



\bibitem[G62]{GRa} Grauert, H.: \"Uber Modifikationen und exzeptionelle
analytische Mengen,   Math. Ann., {\bf 146} (1962), 331--368.






\bibitem[H87]{Hu1987} Huneke, C.: Hilbert functions and symbolic powers,
{\it Michigan Math. Journal} {\bf 34}(2) (1987), 293--318.


\bibitem[HS06]{HS} Huneke, C. and Swanson, I.: Integral closure of ideals, rings, and modules,
London Math. Soc. Lecture Notes Series, vol. 336, Cambridge Univ. Press, Cambridge, 2006.







\bibitem[Kl05]{Kl} Kleiman, St. L.: The Picard scheme, in
`Fundamental Algebraic Geometry: Grothendieck’s FGA Explained',
Mathematical Surveys and Monographs
Volume: 123; 2005, 248--333.

\bibitem[Kl13]{Kleiman2013} Kleiman, St. L.: The Picard Scheme, In `Alexandre
 Grothendieck: A Mathematical Portrait', International Press of Boston, Inc., 2014
 (L. Schneps editor).



\bibitem[La72]{Laufer72} Laufer, H.B.: On rational singularities,
{\em Amer. J. of Math.}, {\bf 94}, 597-608, 1972.




\bibitem[La77]{Laufer77} Laufer, H.B.: On minimally elliptic singularities,
{\em Amer. J. of Math.} {\bf 99} (1977), 1257--1295.






\bibitem[Li69]{Li} Lipman, J.: Rational singularities, with applications to algebraic surfaces and unique factorization, Inst. Hautes \'Etudes Sci. Publ. Math. {\bf 36} (1969), 195--279.










\bibitem[NR]{R} Nagy, J:
Invariants of relatively generic structures on normal surface singularities,
arXiv:1910.03275

\bibitem[N19]{Nagy-holes} Nagy, J.:
Holes in possible values of $h^1$ and geometric genus, arXiv arXiv:1911.07300.




\bibitem[NNI]{NNI} Nagy, J., N\'emethi, A.:
The Abel map for surface singularities  I. Generalities and  examples,
{\it Mathematische Annalen} {\bf 375}(3) (2019), 1427--1487.


\bibitem[NNII]{NNII} Nagy, J., N\'emethi, A.:
The Abel map for surface singularities  II. Generic analytic structure,
{\it Adv. in Math.} {\bf 371} (2020).


 \bibitem[NNIII]{NNIII} Nagy, J., N\'emethi, A.:
The Abel map for surface singularities  III. Elliptic germs,  arXiv:1902.07493.

 \bibitem[NNIV]{NNIV} Nagy, J., N\'emethi, A.:
 The dimension of the image of the Abel map associated with normal surface singularities,
arXiv:1909.07023.






\bibitem[N99b]{Nfive} N\'emethi, A.: Five lectures on normal surface singularities,
lectures at the Summer School in {\em Low dimensional topology} Budapest,
Hungary, 1998; Bolyai Society Math. Studies {\bf 8} (1999), 269--351.


\bibitem[N07]{trieste} N\'emethi, A.: Graded roots and singularities,
{\em Singularities in geometry and topology},  World
Sci. Publ., Hackensack, NJ (2007), 394--463.


\bibitem[N12]{NCL} N\'emethi, A.: The cohomology of line bundles of splice--quotient singularities,
{\em Advances in Math.} {\bf 229} 4 (2012), 2503--2524.












\bibitem[O08]{Ok} Okuma, T.: The geometric genus of splice--quotient singularities,
{\em Trans. Amer. Math. Soc.} {\bf 360} 12 (2008), 6643--6659.


\bibitem[O15]{OkMult} Okuma, T.:
The multiplicity of abelian covers of splice quotient singularities, {\it
Math. Nachrichten } {\bf 288} (2-3) (2015), 343--352.


\bibitem[O17]{O17}
Tomohiro Okuma, \emph{Cohomology of ideals in elliptic surface singularities},
  Illinois J. Math. \textbf{61} (2017), no.~3-4, 259--273.


 \bibitem[O19]{O19} Okuma, T.:
 Normal reduction numbers of normal surface singularities, arXiv:1911.09341.

\bibitem[OWY14]{OWY14} Okuma, T.,   Watanabe Kei-ichi,  Yoshida Ken-ichi:
\emph{Good ideals and {$p_g$}-ideals in two-dimensional normal
  singularities}, Manuscripta Math. \textbf{150} (2016), no.~3-4, 499--520.
arXiv:1407.1590.

\bibitem[OWY15a]{OWY15a} Okuma, T.,   Watanabe Kei-ichi,  Yoshida Ken-ichi:
\emph{Rees algebras and {$p_g$}-ideals in a two-dimensional normal
  local domain}, Proc. Amer. Math. Soc. \textbf{145} (2017), no.~1, 39--47.
 arXiv:1511.00827.

 \bibitem[OWY15b]{OWY15b} Okuma, T.,   Watanabe Kei-ichi,  Yoshida Ken-ichi:
\emph{A characterization of two-dimensional rational singularities via
  core of ideals}, J. Algebra \textbf{499} (2018), 450--468.
arXiv:1511.01553.

\bibitem[OWY19a]{OWY19a}
Okuma, T.,   Watanabe Kei-ichi,  Yoshida Ken-ichi: \emph{Normal reduction numbers for normal surface singularities with application to elliptic singularities of {B}rieskorn type}, Acta Math.
  Vietnam. \textbf{44} (2019), no.~1, 87--100.

\bibitem[OWY19b]{OWY19b}
Okuma, T.,   Watanabe Kei-ichi,  Yoshida Ken-ichi:
\emph{The normal reduction number of two-dimensional cone-like singularities},
arXiv:1909.13190.

 \bibitem[O19]{Okuma19} Okuma, T.:
 Normal reduction numbers of normal surface singularities, arXiv:1911.09341.




\bibitem[Re97]{MR}  Reid, M.: Chapters on Algebraic Surfaces.
In: Complex Algebraic Geometry,
IAS/Park City Mathematical Series,  Volume {\bf 3}  (J. Koll\'ar editor),
3-159, 1997.



\bibitem[Wa70]{Wa70} Wagreich, Ph.: Elliptic singularities of surfaces, {\it Amer. J. of Math.},
{\bf 92} (1970), 419--454.





\end{thebibliography}
\end{document}